\documentclass[12pt,reqno]{amsart}

\usepackage{amssymb,enumerate}

\usepackage[all]{xy}

\numberwithin{equation}{section}

\newcommand{\modules}[1]{#1\text{-}\mathbf{mod}}
\newcommand{\comodules}[1]{\mathbf{comod}\text{-}#1}
\newcommand{\modulesc}[1]{#1\text{-}\mathbf{modc}}

\newtheorem{theorem}{Theorem}[section]
\newtheorem{proposition}[theorem]{Proposition}
\newtheorem{lemma}[theorem]{Lemma}
\newtheorem{corollary}[theorem]{Corollary}
\newtheorem{question}[theorem]{Question}

\theoremstyle{definition}

\newtheorem{remark}[theorem]{Remark}
\newtheorem{example}[theorem]{Example}

\begin{document}

\title[Monodromy of holomorphic differential systems]{On the monodromy of holomorphic differential systems}

\author[I. Biswas]{Indranil Biswas}

\address{Department of Mathematics, Shiv Nadar University,
NH91, Tehsil Dadri, Greater Noida,
Uttar Pradesh 201314, India}

\email{indranil.biswas@snu.edu.in, indranil29@gmail.com}

\author[S. Dumitrescu]{Sorin Dumitrescu}

\address{Universit\'e C\^ote d'Azur\\ CNRS, LJAD, France}

\email{dumitres@unice.fr}

\author[L. Heller]{Lynn Heller}

\address{Beijing Institute of Mathematical Sciences
and Applications, Yanqi Island,
Huairou District, Beijing 101408, P. R. China}

\email{lynn@bimsa.cn}

\author[S. Heller]{Sebastian Heller}

\address{Beijing Institute of Mathematical Sciences
and Applications, Yanqi Island,
Huairou District, Beijing 101408, P. R. China}

\email{sheller@bimsa.cn}

\author[J. P. D. Santos]{Jo\~ao Pedro dos Santos}

\address{Institut Montp\'ellierain Alexander Grothendieck, Universit\'e de Montpellier,
Place Eug\`ene Bataillon, 34090 Montpellier, France}

\email{joao\_pedro.dos\_santos@yahoo.com}



\begin{abstract}
First we survey and explain the strategy of some 
recent results that construct holomorphic $\text{sl}(2, \mathbb C)$--differential systems over 
some Riemann surfaces $\Sigma_g$ of genus $g \,\geq\, 2$, satisfying the
condition that the image of the associated 
monodromy homomorphism is (real) Fuchsian \cite{BDHH} or some cocompact Kleinian subgroup 
$$\Gamma \, \subset \, \text{SL}(2, \mathbb C)$$ as in \cite{BDHH2}. As a consequence, 
there exist holomorphic maps from $\Sigma_g$ to the quotient space $\text{SL}(2, \mathbb C)/ 
\Gamma$, where $\Gamma\, \subset\, \text{SL}(2, \mathbb C)$ is a cocompact lattice,
that do not factor through any elliptic curve \cite{BDHH2}. This answers positively 
a question of Ghys in \cite{Gh}; the question was also raised by Huckleberry and Winkelmann
in \cite{HW}.

Then we prove that when $M$ is a Riemann surface, a Torelli type theorem holds for the affine group scheme over $\mathbb C$ obtained from the category of holomorphic connections on {\it \'etale trivial} holomorphic bundles.

After that, we explain how to compute in a simple way the holonomy of a holomorphic connection 
on a free vector bundle.
 
Finally, for a compact K\"ahler manifold $M$, we investigate the neutral Tannakian category given by 
the holomorphic connections on \'etale trivial holomorphic bundles over $M$. If $\varpi$ 
(respectively, $\Theta$) stands for the affine group scheme over $\mathbb C$ obtained from the 
category of connections (respectively, connections on free (trivial) vector bundles), then the natural 
inclusion produces a morphism $v\,:\,{\mathcal O}(\Theta)\,\longrightarrow\,
{\mathcal O}(\varpi)$ of Hopf algebras. We present a description of the transpose of $v$ in
terms of the iterated integrals.
\end{abstract}

\maketitle

\section{Introduction}

In recent times, there has been a resurgent interest in the theory of holomorphic {\it 
differential systems}, or connections on trivial bundles, in the special case of a {\it 
compact} ambient space. This interest has been mainly triggered by a strategy of Ghys, 
presented in \cite{CDHL}, to answer a question raised first by Huckleberry and Winkelmann in 
\cite{HW} and later by Ghys in \cite{Gh}. This question in \cite{HM, Gh} inquires about the 
existence of holomorphic maps from compact hyperbolic Riemann surfaces to compact complex 
manifolds of the form $\mathrm{SL}(2,\mathbb C)/\Gamma$, with $\Gamma$ a lattice in 
$\mathrm{SL}(2,\mathbb C)$ (which does not factor to any elliptic curve). The suggested 
direction in \cite{CDHL} towards an answer is the study of representations of surface groups 
into $\mathrm{SL}(2,\mathbb C)$ giving rise to rank two {\it free} (trivial) holomorphic 
vector bundles over some Riemann surface of genus $g \,\geq\, 2$ (see Question 
\ref{12.10.2022--2}).

Although the theory of differential systems (holomorphic connections on free vector bundles) 
has always been a much studied subject, the restrictions imposed by a compact ambient space 
are rather pleasing and produce fruitful geometric situations. The present paper is a 
contribution to this theme, containing an extended survey presenting different backgrounds and recent advances.

To motivate the reader and display some of the beautiful geometry we are set to explain,
we begin by surveying some recent results on the theory of holomorphic $\text{SL}(2, \mathbb C)$--differential systems over compact Riemann surfaces and their monodromy representations.
This is done in Section \ref{Sect diff systems}, 
where after reviewing Ghys' questions and its relation to the theory of differential systems (see Question \ref{12.10.2022--1} and Question \ref{12.10.2022--2}), we move on to explain some local results about the
Riemann-Hilbert mapping, proved in \cite{CDHL, BD2}. (It is perhaps worth noting that the main results of these works were recently extended to the logarithmic case in \cite{ABDH}.) 

Despite the local results obtained in \cite{CDHL}, for curves of genus $g\,=\,2$, and then in 
\cite{BD2, ABDH}, Ghys' question was still open until very recently. In particular, it was not 
known whether one can realize given real, discrete or Zariski dense subgroups in $\text{SL}(2, 
\mathbb{C})$ as monodromy of $\text{SL}(2, \mathbb{C})$--local systems over some Riemann 
surface $\Sigma_g$ of genus $g \geq 2$.

An important step toward the understanding of the image of the monodromy homomorphism was very 
recently realized in \cite{BDHH} where holomorphic connections with (real) Fuchsian monodromy 
were constructed. This completely answers a question asked in \cite{CDHL}.
 
Eventually Ghys strategy has been successfully implemented in the very recent work \cite{BDHH2}. It is first shown in \cite{BDHH2} that every irreducible $\text{SL}(2, \mathbb 
R)$--representation with sufficient many symmetries can be realized as the monodromy of a 
holomorphic irreducible $\text{SL}(2, \mathbb C)$--connection on the rank two trivial 
holomorphic bundle over some (very symmetric) Riemann surface $\Sigma$ such that 
$\text{genus}(\Sigma)\, \geq\, 2$. Then an example of such a symmetric representation 
contained in a cocompact lattice $\Gamma$ in $\text{SL}(2, \mathbb C)$ is given and a 
holomorphic map $$\Sigma \, \longrightarrow \text{SL}(2, \mathbb C)/ \Gamma$$ is constructed 
(which does not descend on any elliptic curve), answering the open question of Ghys \cite{Gh} 
and Huckleberry-Winkelmann \cite{HW}.

Once light was thrown on the pertinence of the theory of differential systems on compact 
spaces, it became clear, see for example \cite{BDDH}, that a reasonable way to study 
holomorphic connections is to impose conditions on the underlying vector bundle. It is with 
this idea in mind that we go over to Section \ref{Sect finite}. There, we give ourselves a 
compact complex manifold $M$ and study integrable connections where the underlying vector 
bundle enjoys the following property: it becomes holomorphically trivial after an \'etale 
covering or, it is {\it \'etale trivial}. (We explain at the start of Section \ref{Sect 
finite} that the latter property can be recovered from the existence of a connection with 
finite monodromy.) We start by reviewing recent work \cite{Bi,BD2} which connects \'etale 
trivial bundles with {\it finite} \cite[p. 35]{No1} and {\it ``properly'' trivial} ones 
\cite[Property (T), p. 225]{BdS}. Then, we follow the path opened by Tannakian categories and 
define an affine group scheme whose representations amount to the aforementioned connections.

Another of the pleasant properties of differential systems (in compact ambient) is the 
simplicity in which differential Galois groups can be computed, as was the theme in 
\cite{biswas-hai-dos_santos22}. (This paper takes the matter from the purely algebro-geometric 
viewpoint.) Section \ref{Sect Trivial bundle} deals with an analogue of 
\cite{biswas-hai-dos_santos22} in a complex analytic setting by showing that the smallest 
reduction of structure group preserving a connection is given by the algebraic hull of the 
``connection matrices''.

In Section \ref{Sect iterated}, we study the relation between the category of integrable 
connections and the category of integrable connections on trivial vector bundles by the 
Tannakian point of view, as in \cite{BDDH}. Moreover, this study is done by putting together 
Hochschild-Mostow's \cite[p.1140]{hochschild-mostow57} (as understood by \cite{sweedler69}) 
and Saavedra-Grothendieck's \cite{Sa} approach to Tannaka duality. Recall that the first one 
is achieved by means of ``representative functions'' while the second is, from the start, 
categorical. In explicit terms: Let $M$ be a compact complex manifold, $x_0$ a point in $M$, 
$\mathcal C_{\rm dR}(M)$ and $\mathcal T(M)$ respectively the categories of integrable 
connections and integrable connections on trivial vector bundles. As is well-known, $\mathcal 
C_{\rm dR}(M)$ is equivalent to the category of $\mathbb C\pi_1(M,\,x_0)$--modules while, as 
argued in \cite{biswas-hai-dos_santos22}, $\mathcal T(M)$ is equivalent to the category of 
representations of a certain co-commutative Hopf algebra \[\mathfrak 
A_M\,=\,\frac{\text{Tensor algebra on $H^0(M,\, \Omega_M^1)^*$}}{\text{a certain ideal}}.\] 
Following \cite[Chapter VI]{sweedler69}, these categories are then, respectively, equivalent 
to the categories of comodules over the Hopf-algebras $\mathbb C\pi_1(M,\,x_0)^\circ$ and 
$\mathfrak A_M^\circ$. Hence, the morphism $\mathfrak A_M^\circ\,\longrightarrow\, \mathbb 
C\pi_1(M,\,x_0)^\circ$ coming from the inclusion $\mathcal T(M)\,\longrightarrow\,\mathcal C_{\rm dR}(M)$ through \cite{Sa} gives rise to an arrow 
between certain completions:
\[u:
\mathbb C\pi_1(M,\,x_0)\,\widehat{\,\,}\,\, \longrightarrow\,\,\,
\widehat{\mathfrak A}_M;
\]
it is rendering explicit the composition $\mathbb C\pi_1(M,\,x_0)\,\longrightarrow\, \mathbb 
C\pi_1(M,\,x_0)\,\widehat{\,\,} \longrightarrow\, \widehat{\mathfrak A}_M$ in terms of 
iterated integrals that occupies Section \ref{Sect iterated}. In doing so, we recover 
expressions appearing in the works of Chen, Parshin and Hain but, in our context, the source 
and target of the map $u$ are bigger completions and carry more information. To wit, these are 
the pro-finite {\it dimensional} (as opposed to pro-{\it finite}) completions.

This work is dedicated to Oscar Garc\'{\i}a-Prada on the occasion of his sixtieth birthday.

\section{Curves in compact quotients of $\rm{SL}(2, \mathbb C)$}\label{Sect diff systems} 

A natural class of compact complex manifolds generalizing complex tori are complex manifolds 
whose holomorphic tangent bundle is holomorphically trivial. By a classical result of Wang 
\cite{Wa}, compact complex manifolds with holomorphically trivial tangent bundle are 
biholomorphic to the complex manifolds of the form $G/\Gamma$, where $G$ is a complex Lie 
group and $\Gamma\, \subset\, G$ a discrete cocompact subgroup. Compact complex manifolds 
in this class are called {\it (complex) parallelizable manifolds}. They are in general not 
K\"ahler. More precisely, a complex parallelizable manifold $G/\Gamma$ is K\"ahler if and only 
if the complex Lie group $G$ is abelian, in which case $G/\Gamma$ is biholomorphic to a 
compact compact complex torus \cite{Wa}.

In particular, when $G$ is a complex semi-simple Lie group, and $\Gamma\, \subset\, G$ is a cocompact lattice, the complex 
parallelizable manifold $G/ \Gamma$ is non-K\"ahler. Moreover, a theorem due to Huckleberry and Margulis \cite{HM} says that $G / 
\Gamma$ does not admit any complex analytic hypersurface. This implies, in particular, that $G / \Gamma$ does not admit any 
nonconstant meromorphic function and, consequently, the algebraic dimension of the complex parallelizable manifold $G/\Gamma$ is zero.

When $G$ is a complex semi-simple Lie group with no local factor isomorphic to $\text{SL}(2, \mathbb C)$ (for example, $G=\text{SL}(n, 
\mathbb C)$, with $n\,\geq\, 3$), a vanishing result of Raghunathan, \cite{Ra}, says that the complex structure of any compact quotient 
$G/ \Gamma$, where $\Gamma$ a discrete cocompact subgroup in $G$, is {\it rigid}: any complex structure on the underlying real manifold $G/ 
\Gamma$ close to the standard one (endowed by the complex structure of $G$) actually coincides with the standard one.

The remaining question about the rigidity (or flexibility) of the standard complex structure on the compact quotients $G/ \Gamma$, 
with $G\,=\,\text{SL}(2, \mathbb C)$, was studied by Ghys in \cite{Gh} (see also \cite{Raj}). In this situation Ghys computes in \cite{Gh} 
the Kuranishi space of $\text{SL}(2, \mathbb C) / \Gamma$ and proves that this deformation space is nontrivial for any discrete 
cocompact subgroup $\Gamma \, \subset \, \text{SL}(2, \mathbb C)$ with positive first Betti number. This implies that the standard 
complex structure of these compact quotients of $\text{SL}(2, \mathbb C)$ is flexible and admit nontrivial deformations
described by Ghys.

It may be mentioned that the discrete cocompact subgroups $\Gamma \, \subset \, \text{SL}(2, \mathbb C)$ with positive first Betti 
number are quite abundant (see \cite{La}); their existence is closely related to the compact quotients of the $3$--hyperbolic space 
${\mathbb H}^3$ (i.e. the compact hyperbolic $3$--manifolds). More precisely, the complex Lie group $\text{PSL}(2, \mathbb{C})$ being the 
orientation preserving isometry group of the 3-hyperbolic space ${\mathbb H}^3$, it is identified with the oriented orthonormal frame 
bundle of ${\mathbb H}^3$. Consequently, the orthonormal frame bundle of any compact hyperbolic $3$--manifold $M$ (diffeomorphic to a 
quotient ${\mathbb H}^3/ \Gamma$, with $\Gamma$ a lattice in $\text{SL}(2, \mathbb{C})$) admits the structure of a compact complex 
parallelizable manifold which is biholomorphic to a quotient of $\text{SL}(2, \mathbb{C})$ by its cocompact lattice $\Gamma$ 
(isomorphic to the fundamental group of $M$). One of the reasons for the complex geometry of the compact quotients of $\text{SL}(2, 
\mathbb{C})$ to be of interest is this fact.

As mentioned before, the flexibility of the complex structure of $\text{SL}(2, \mathbb C)/ 
\Gamma$ was discovered by Ghys in \cite{Gh} (see also \cite{Raj}) where he showed that the 
corresponding Kuranishi space has positive dimension for all $\Gamma$ with positive first 
Betti number. Notice that compact hyperbolic $3$--manifolds with prescribed rational 
cohomology ring (in particular, with arbitrarily large first Betti number) can be constructed 
using Thurston's hyperbolisation Theorem (see \cite[Lemme 6.2]{Gh}). It should be mentioned 
that the complex geometry of the orthonormal frame bundle of the compact hyperbolic 
$3$--manifold ${\mathbb H}^3/ \Gamma$ endowed with the complex structure $\text{SL}(2, \mathbb 
C)/ \Gamma$ is not well understood yet.

When studying the geometry of compact quotients of $\text{SL}(2, \mathbb C)$, Ghys asked in \cite{Gh} the following question, also raised 
by Huckleberry and Winkelmann in \cite{HW}:

\begin{question}\label{12.10.2022--1}
 Does there exist a compact quotient ${\rm SL}(2, \mathbb C)/ \Gamma$ 
admitting a compact holomorphic curve of genus $g \,\geq\, 2$?
\end{question}

A negative answer to this question would generalize the Huckleberry-Margulis Theorem \cite{HM} 
(for $G\,=\,\text{SL}(2, \mathbb C)$) to compact holomorphic curves of genus $g \geq 2$. 
Notice that, in the case $g\,=\,1$, elliptic curves covered by one-parameter subgroups in 
$\text{SL}(2, \mathbb C)$ are known to exist in certain quotients $\text{SL}(2, \mathbb C)/ 
\Gamma$.

In an attempt to answer positively Question \ref{12.10.2022--1}, Ghys developed the following 
strategy to characterize compact holomorphic curves of genus $g\, \geq\, 2$ in quotients 
$\text{SL}(2, \mathbb C)/ \Gamma$ (see \cite{CDHL} where this strategy is explained in 
detail).

{\it Find a Riemann surface $\Sigma$ of genus $g \geq 2$ and an irreducible holomorphic ${\rm 
SL}(2, \mathbb C)$--connection on the rank two trivial holomorphic bundle $\Sigma \times 
\mathbb{C}^2$ over $\Sigma$ such that the image of the corresponding monodromy homomorphism 
lies in a cocompact lattice $\Gamma$ in ${\rm SL}(2, \mathbb C)$.}

Recall that a holomorphic connection on a holomorphic bundle over a Riemann surface is 
automatically flat.

As Ghys observed, in this context, the parallel frame of the above (flat) holomorphic 
connection gives rise to a holomorphic map from $\Sigma$ into the quotient $\text{SL}(2, 
\mathbb C)/ \Gamma$ (which, due to the irreducibility of the connection, does not factor 
through any elliptic curve) (see \cite{CDHL} for more details). More generally, Ghys asked:

\begin{question}\label{12.10.2022--2} Given a compact orientable surface $X$ of genus $g \,\geq \,2$, describe the space of all
irreducible group homomorphisms $\rho \,:\, \pi_1(X)\,
\longrightarrow\,{\rm SL}(2,\mathbb{C})$ for which there exists a complex structure $\Sigma$ on
$X$ such that the holomorphic vector bundle over $\Sigma$ defined by $\rho$ is holomorphically trivial.
\end{question}

If some irreducible homomorphism $\rho \,:\, \pi_1(X)\, \longrightarrow\,{\rm 
SL}(2,\mathbb{C})$ satisfying the above condition has the property that $\rho(\pi_1(X))$ 
 is contained in a cocompact lattice $\Gamma \,\subset\, \text{SL}(2, \mathbb{C})$, then 
the local system on the Riemann surface $\Sigma$ given by $\rho$ produces a nontrivial 
holomorphic map from $\Sigma$ to $\text{SL}(2, \mathbb{C}) / \Gamma$ (more details about this 
method can be found in \cite{CDHL}).

A first local result in the direction of Question \ref{12.10.2022--2} was obtained in 
\cite{CDHL} for the genus $g=2$. Other local results generalizing the one in \cite{CDHL} were 
obtained in \cite{BD2}, and then in \cite{ABDH} for logarithmic differential systems. We 
briefly recall below those local results obtained in \cite{CDHL,BD2}.

The aim of both \cite{CDHL, BD2} is to study the Riemann-Hilbert correspondence for 
differential systems over compact Riemann surfaces of genus $g \,\geq\, 2$. Denote by $X$ a 
compact connected oriented topological surface of genus $g \,\geq\, 2$. We also denote by $G$ 
a connected reductive affine algebraic group defined over the field of complex numbers.

Let us now fix a complex structure $\Sigma$ on $X$ (equivalently, it is the choice of an 
element in the Teichm\"uller space of $X$). We consider also a holomorphic connection $\nabla$ 
on the trivial holomorphic principal $G$--bundle $\Sigma \times G$ over $\Sigma$. Notice that 
$\nabla$ is determined by an element $\delta\, \in\, {\mathfrak g}\otimes H^0(\Sigma, \, 
K_\Sigma ),$ where $\mathfrak g$ is the complex Lie algebra of $G$ and $K_\Sigma $ is the 
canonical bundle of $\Sigma $. Fix a base point $x_0\, \in\, X$ and consider the corresponding 
universal cover $\pi \,:\, \widetilde{\Sigma} \,\longrightarrow\, \Sigma$ of $\Sigma$ 
(endowed with the induced complex structure). On the trivial principal $G$--bundle 
$\widetilde{\Sigma}\times G$ over $\widetilde \Sigma$ we have the pulled back holomorphic 
(flat) connection $\pi^*\nabla$.

For any locally defined $\nabla$--parallel section $\phi$ of $\Sigma \times G$, the pulled 
back local section $\pi^*\phi$ of $\widetilde{\Sigma}\times G$ extends to a 
$\pi^*\nabla$--parallel section on entire $\widetilde \Sigma$. The above extension of 
$\pi^*\phi$ identifies with a holomorphic map $\widetilde{\Sigma} \,\longrightarrow\, G$ which 
is $\pi_1(X ,\, x_0)$--equivariant with respect to the natural action of $\pi_1(X,\, x_0)$ on 
$\widetilde \Sigma$ through deck transformations and the action of $\pi_1(X,\, x_0)$ on $G$ 
through a group homomorphism $\pi_1(X,\, x_0) \,\longrightarrow\, G$, which is known as the 
{\it monodromy } (homomorphism) of the flat connection $\nabla$.

The above monodromy homomorphism depends on the choice of the holomorphic trivialization of 
the principal $G$--bundle over $\widetilde{\Sigma}$. Nevertheless, the corresponding element 
of the character variety of $G$--representations $$\Xi \,:=\, \text{Hom}(\pi_1(X),\, 
G)/\!\!/G$$ does not depend either on the choice of the trivialization of the principal 
$G$--bundle over $\widetilde{\Sigma}$, or on the choice of the base point $x_0$ (or on the 
choice of the local section $\phi$).

The character variety $\Xi$ is known to be a singular complex analytic space of complex 
dimension $2((g-1)\cdot\dim\, [G,\, G] + g\cdot (\dim G - \dim\, [G,\, G]))$; see, for 
example, \cite[Proposition 49]{Sik}. The dimension of $\Xi$ equals $2(g-1)d+2gc$, with $d$ 
being the complex dimension of the commutator group $ [G,\, G]$ for $G$ and $c$ being equal to 
$\dim G - \dim\, [G,\, G]$.

As in \cite{CDHL,BD2}, denote by $\text{Syst}$ the space of all differential systems, meaning 
the space of the pairs $(\Sigma,\, \nabla)$, where $\Sigma$ is a complex structure on $X$ (an 
element of the Teichm\"uller space for $X$) and $\nabla$ is a holomorphic connection on the 
trivial holomorphic principal $G$--bundle $\Sigma \times G$ over $\Sigma$. This space of 
differential systems on $X$ is a singular complex analytic space of dimension $(g-1)(d+3)+gc$. 
Associating to any pair $(\Sigma,\, \nabla) \,\in\, \text{Syst}$ the element in the character 
variety $\Xi$ corresponding to the monodromy morphism of $\nabla$, we obtain a holomorphic map 
from $\text{Syst}$ to the character variety $\Xi$. This map is the restriction to 
$\text{Syst}$ of the Riemann--Hilbert correspondence.

We consider now the nontrivial Zariski open subset $\text{Syst}^{\text{irred}}$ in 
$\text{Syst}$ defined by all pairs $(\Sigma,\, \nabla)$ such that the holomorphic connection 
$\nabla$ is {\it irreducible} (i.e., the monodromy homomorphism associated to $\nabla$ does 
not factor through any proper parabolic subgroup of $G$). The Zariski open set 
$\text{Syst}^{\text{irred}}$ is a smooth complex manifold (see, for example, Corollary 50 in 
\cite{Sik}). The image of $\text{Syst}^{\rm irred}$ under the Riemann--Hilbert correspondence 
described above lies in the smooth Zariski orbifold open subset $$\Xi^{\rm irred}\, \subset\, 
\Xi$$ defined by the irreducible homomorphisms $\pi_1(X)\, \longrightarrow\, G$ (i.e., 
homomorphisms that do not factor through some proper parabolic subgroup of $G$).

Denote by $\text{Mon}$ (as monodromy) this restriction of the Riemann--Hilbert correspondence to $\text{Syst}^{\text{irred}}$:
\begin{equation}\label{mm}
\text{Mon} \,:\, \text{Syst}^{\text{irred}} \,\longrightarrow\, \Xi^{\rm irred}
\end{equation}

Then $\text{Mon}$ is a holomorphic map between the complex manifolds 
$\text{Syst}^{\text{irred}}$ and $\Xi^{\rm irred}$.

The main result proved in \cite{BD2} is the following:

\begin{theorem}[\cite{BD2}]\label{thi1}
If the complex dimension of $G$ is at least three, the map ${\rm Mon}$
in \eqref{mm} is an immersion at the generic point.
\end{theorem}

If $G\,=\, {\rm SL}(2,{\mathbb C})$, the dimensions of
$\text{Syst}^{\rm irred} $ and $\Xi^{\rm irred}$ are both $6g-6$. Therefore, 
in that case Theorem \ref{thi1} implies that $\text{Mon}$ is a
local biholomorphism at the generic point.
It should be mentioned that examples constructed in \cite{CDHL} show that for
$G\,=\, {\rm SL}(2,{\mathbb C})$ and $\Sigma$ of genus $g \,\geq\, 3$, the monodromy map $\text{Mon}$ 
is not always a local biholomorphism (over the entire $\text{Syst}^{\text{irred}}$).

When $G\,=\, {\rm SL}(2,{\mathbb C})$ and $g\,=\,2$, the main result of \cite{CDHL} says that the 
map $\text{Mon}$ in \eqref{mm} is a local biholomorphism over entire $\text{Syst}^{\text{irred}}$. 
An alternative proof of this result of \cite{CDHL} was also provided in \cite[Corollary 5.6]{BD2}.

The main motivation for the authors of \cite{CDHL} and \cite{BD2} was the case $G\,=\, {\rm 
SL}(2,{\mathbb C})$ because Ghys strategy relates the monodromy of ${\mathfrak s}{\mathfrak 
l}(2, {\mathbb C})$--differential systems to the question of existence of holomorphic curves 
of genus $g \,>\,1$ lying in quotients of ${\rm SL}(2,{\mathbb C})$ by cocompact lattices.

It was asked in \cite{CDHL} if there exist holomorphic ${\mathfrak s}{\mathfrak l}(2, {\mathbb 
C})$--differential systems on Riemann surfaces of genus $g \,>\,1$ having real or discrete 
monodromy. Indeed, despite those local results obtained in \cite{CDHL}, for curves of genus 
$g\,=\,2$, and then in \cite{BD2, ABDH}. Ghys' question was still open until very recently. In 
particular, it was not known whether it is possible to realize any discrete or Zariski dense 
subgroup in $\text{SL}(2, \mathbb{C})$ as the monodromy of a $\text{SL}(2, \mathbb{C})$--local 
system for some complex structure $\Sigma$ on $X$.

A first answer to this question was given in \cite{BDH} where the main result is a construction of an irreducible holomorphic connection 
with ${\rm SL}(2,\mathbb R)$--monodromy on the rank two trivial holomorphic vector bundle over a compact Riemann surface of genus $g 
\,>\,1$.

Another important step towards realizing Ghys' strategy was very recently made in \cite{BDHH} 
where holomorphic connections with (real) Fuchsian monodromy were constructed. The result is 
the following:

\begin{theorem}[\cite{BDHH}]\label{thm:bdhh}
For every integer $g\,\geq \, 2$, there exists a (hyperelliptic) Riemann surface $\Sigma_g$ of genus $g$, such 
that the rank two trivial holomorphic vector bundle $\Sigma_g\times \mathbb{C}^2$
over $\Sigma_g$ admits infinitely many holomorphic ${\rm SL}(2, \mathbb C)$--connections 
with Fuchsian monodromy representation.
\end{theorem}

The above result completely answers the question asked in \cite{CDHL} by Calsamiglia, Deroin,
Heu and Loray.

In this context it is also natural to ask which holomorphic rank two bundles over a given 
Riemann surface $\Sigma$ admit holomorphic connections with Fuchsian monodromy 
representations. This question was indeed raised by Katz in \cite[p.~555--556]{Ka} (where the 
question is attributed to Bers) in 1978 and is still unsolved. Even when restricting to the 
trivial rank two holomorphic bundle, it was not known before \cite{BDHH} whether a holomorphic 
connection $\nabla$ with Fuchsian monodromy representation exists.

Eventually, Ghys strategy was successfully realized in the recent work \cite{BDHH2}. In 
\cite{BDHH2} it is first shown that every irreducible $\text{SL}(2, \mathbb 
R)$--representation with sufficient many symmetries can be realized as the monodromy of a 
holomorphic irreducible $\text{SL}(2, \mathbb C)$--connection on the rank two trivial 
holomorphic bundle over some (very symmetric) Riemann surface $\Sigma$ such that 
$\text{genus}(\Sigma)\, \geq\, 2$. Then an example of such a symmetric representation with 
image contained in a cocompact lattice $\Gamma$ in $\text{SL}(2, \mathbb C)$ is constructed, 
with $\Gamma$ being by the dodecahedron tiling $\{5,\,3,\,4\}$ of the hyperbolic 3-space 
${\mathbb H}^3$. The results in \cite{BDHH2} imply the following:

\begin{theorem}[\cite{BDHH2}]\label{thm:bdhh2}
Let $\Gamma$ be the cocompact lattice in ${\mathbb H}^3$ given by the dodecahedron tiling $\{5,\,3,\,4\}$ of
the hyperbolic 3-space. Then there exists a non constant holomorphic map from a compact curve
$\Sigma$ of genus $4$ in ${\rm SL}(2, \mathbb C)/\Gamma.$
\end{theorem}

\begin{remark}
The holomorphic map from $\Sigma$ to SL$(2, \mathbb C)/\Gamma$, constructed in
Theorem \ref{thm:bdhh2}, has 4 simple branch points and by Riemann-Hurwitz Theorem it cannot factor through a lower genus surface.
\end{remark}

This gives an affirmative answer to above question raised by
Huckleberry and Winkelmann \cite{HW} and by Ghys \cite{Gh}.

\subsection{Strategy of the proofs of Theorem \ref{thm:bdhh} and Theorem \ref{thm:bdhh2}}

For both above results, namely Theorem \ref{thm:bdhh} and Theorem \ref{thm:bdhh2}, 
straightforward methods such as computing the monodromy representation explicitly or using 
geometric existence results, such as the uniformization metric for Fuchsian representations on 
the Gunning bundle, is not developed or applicable. Therefore, the general philosophy is to 
start with a given monodromy of a known holomorphic differential system and try to deform it 
until the deformed monodromy has the desired properties. This leaves of course a lot of 
freedom for implementation. Though initial setup of both proofs are the same, Theorem 
\ref{thm:bdhh2} requires a completely new set of further ideas as well as some further 
applications of the ideas in Theorem \ref{thm:bdhh}. Hence, we begin with describing the 
general setup as well as the strategy for Theorem \ref{thm:bdhh}.

The first step is to reduce complexity. We first impose various symmetries of the underlying 
Riemann surface and the considered holomorphic differential systems as follows. For every 
genus $g$ consider the complex 1-dimensional family of Riemann surfaces which admit a $\mathbb 
Z_{g+1}$--action with 4 fixed points of order $g,$ i.e, the Riemann surface is given by a 
$(g+1)-$fold covering of the complex projective line totally branched over 4 points. Then the 
connected components of $\mathbb Z_{g+1}$--equivariant connections are complex 2-dimensional 
subspaces of the de Rham moduli space of flat $\mathrm{SL}(2,\mathbb C)$--connections. The 
Fuchsian representations of the symmetric Riemann surfaces automatically give rise to $\mathbb 
Z_{g+1}$--equivariant connections. By \cite{biswas97}, all these equivariant connections are 
determined by logarithmic $\mathrm{SL}(2,\mathbb C)$--connections on the 4-punctured sphere 
with prescribed local conjugacy classes depending (only) on both the genus of the surface and 
the connected component of $\mathbb Z_{g+1}$--equivariant connections inside the de Rham 
moduli space.

For example, the eigenvalues of the residues of a logarithmic connection over the complex 
projective line have to be $\pm\frac{g}{2(g+1)}$, at each of the 4 singular points, in order 
to correspond to an equivariant Fuchsian representation, see \cite[Section 3]{BDHH}. Moreover, 
there exist exactly one affine (complex) line of logarithmic connections in each of these 
connected components which correspond to holomorphic systems on the compact Riemann surface of 
genus $g,$ given by a specific parabolic structure (see \cite[Theorem 3.2 (5)]{HHSch} or 
\cite[Proposition 3.1]{BDHH}).

By imposing compatible reality conditions on the Riemann surface and the flat connections respectively, we can reduce further such that 
(the smooth components of) the moduli space of real symmetric Riemann surfaces is real 1-dimensional, and the real subspace of the Betti moduli space is real 2-dimensional. In other words, the considered Riemann surfaces possess a real involution $\tau$ compatible 
with the $\mathbb Z_{g+1}$--action (which is equivalent to the 4 branch points lying on one circle) and the considered connections $\nabla$ lies in the 
same gauge class as $\tau^*\overline{\nabla}$, we refer to $\nabla$ then as a $\tau$--real connection. Considering $\tau-$real connections on real Riemann surfaces leads to certain traces of the monodromies being automatically real. In order for the connection to be called real in following, we require all the monodromy traces along every closed curve to be real.

\subsubsection{Theorem \ref{thm:bdhh}}

Theorem \ref{thm:bdhh} asserts the existence of infinitely many holomorphic systems with Fuchsian monodromy on a real symmetric Riemann 
surface. The first observation is that $\tau$--real connections in the appropriate component (determined by the eigenvalues) give rise to 
Fuchsian representations if it is real and the traces satisfy some additional inequalities. Since the underlying Riemann surface is a 
4-punctured sphere, the character variety, or its Betti moduli space, is determined by 3 traces $x,\,y$ and $z$. For given $x$ and $y$ the 
third one must satisfy the character variety equation, which has two solutions $z$ and $\widetilde z$. Due to $\tau-$symmetry $z$ and 
$\widetilde z$ are automatically real and the reality of $x$ implies then $y$ being real as well. When working with the one-puncture 
torus instead of the 4-punctured sphere we can avoid dealing with further inequalities determining the connected component of Fuchsian 
representations. In other words, in order to find a Fuchsian representation it suffices to show that just one (additional) trace $x$ is 
real on the one-punctured torus.

The trivial holomorphic structure when lifting the rank two flat bundle to the compact Riemann 
surface $\Sigma_g$ of genus $g$ is determined by a specific parabolic structure (see 
\cite[Theorem 3.2 (5)]{HHSch} or \cite[Proposition 3.1]{BDHH}). Since the space of 
($\tau$--symmetric) Higgs fields on the other hand is only (real) one-dimensional, i.e., 
generated by some $\Psi,$ it remains to show that there exist (infinitely many) parameters $t 
\in \mathbb R$ such that $\nabla + t \Psi$ has real $x$, where $\nabla$ is the unique unitary 
connection inducing that parabolic structure. Again, computing the monodromy directly is out 
of reach, but its asymptotic for large $t$ can be studied using $WKB$ analysis. Though the 
method has been carried out successfully in various situations, somehow surprisingly, the 
existing mathematical literature did not cover the asymptotic analysis of the trace of the 
monodromy, at least to the best of our knowledge. It was then shown by Takuro Mochizuki in the 
appendix of \cite{BDHH} that
$$\lim_{t\to\infty}\text{tr}(P_\gamma(\nabla +t \Psi))e^{t\int_\gamma \sqrt{\det\Psi}}\,\,=\,\,c,$$
where $P_\gamma(\nabla +t \Psi)$ is the monodromy of $\nabla+t \Psi$ along the curve $\gamma$
which has to satisfy
\begin{equation}\label{wkb-curve}\Re(\sqrt{\det\Psi}(\gamma'))\,\,<\,\,0,
\end{equation}
and $c\,\neq\,0$ is a non-zero constant. As the curve $\gamma$ corresponding to the trace $x$ 
satisfies \eqref{wkb-curve} and $\int_\gamma \sqrt{\det\Psi}\,\notin\,\mathbb R$ we find 
infinitely many parameters $t_n$ accumulating at $t\,=\,\infty$ for which all monodromies of the 
connections $\nabla+t_n \Psi$ are real and therefore give Fuchsian representations.

\subsubsection{ Theorem \ref{thm:bdhh2}}

While in the space of $\text{SL}(2, \mathbb R)$--representations, the Fuchsian ones form a 
whole connected component, cocompact lattices in $\text{SL}(2, \mathbb C)$ are discrete. 
Therefore, a new set of ideas is needed to attain a specific representation with holomorphic 
connections rather than just reaching the component when varying a parameter. In contrast to 
Theorem \ref{thm:bdhh} where we fixed the Riemann surface type and show existence of countably 
infinite many Fuchsian representations on the trivial holomorphic bundle, we now vary the 
Riemann surface type and show that the representations obtained from holomorphic system 
exhaust the whole (real 1-dimensional) space of symmetric and real representations. Theorem 
\ref{thm:bdhh2} then follows from the fact that the dodecahedron tilling has a sublattice 
inducing such a symmetric and real representation.

In other words, let $\rho$ be an arbitrary real and symmetric representation and let $\rho^F$ 
be a suitable real, $\tau$--real and $\mathbb Z_{g+1}$--equivariant representation induced by 
the monodromy of a holomorphic system on a real, symmetric Riemann surface. Then the aim is to 
deform $\rho^F$ via deforming the underlying (real, symmetric) Riemann surface structure and 
the corresponding holomorphic system until we hit $\rho$.

The existence of the initial value $\rho^F$ is guaranteed by \cite{BDHH}, but the proof of 
\ref{thm:bdhh2} in fact gives an independent proof for the existence of the Fuchsian 
representations of \cite{BDHH}. Local existence and uniqueness of the deformation by varying 
the underlying real symmetric Riemann surface is based on a technical Lemma \cite[Lemma 
5.2]{BDHH2} which was motivated by the main result of \cite{CDHL} for genus 2 surfaces and 
whose proof uses the detailed analysis of the isomonodromic deformations of the Lam\'e 
equation by Loray in \cite{Loray16}. An important observation is the following. We can squeeze 
a real, $\tau$--real and equivariant holomorphic system between the (orbifold) uniformization 
of the Riemann surface and a grafting of the (orbifold) uniformization of another real, 
symmetric Riemann surface inside the ordered space of real, $\tau$--real and equivariant 
representations. The ordering of all three connections is preserved along the deformation, and 
it is shown that the orbifold uniformization and their graftings both cover the whole 
1-dimensional space of real, $\tau$--real and equivariant representations. We should emphasize 
that the induced deformation of the representations is not necessarily continuous introducing 
some technicalities to show the exhaustion property.

The last point to explain here is why and how we can squeeze the holomorphic system with real 
monodromy between a (orbifold) uniformization and a grafting of another uniformization. The 
main ingredient is the new observation that grafting changes the spin structure in a very 
natural way. Using this we can show the existence of a non-trivial path in the moduli space of 
bundles which admit a unique lift to real, $\tau$--real and equivariant connections, which 
starts and ends at two different Gunning bundles (respectively maximally unstable bundles in a 
given component of equivariant connections) and passes through the trivial holomorphic 
structure. This path naturally deforms under the deformation of the underlying Riemann surface 
structure, and as orbifold uniformisations and their graftings exhaust the 1-dimensional space 
of real, $\tau$--real equivariant connections the existence of holomorphic curves of genus 
$g>1$ in certain quotients $\mathrm{SL}(2,\mathbb C)/\Gamma$ by cocompact lattices follows.

\section{\'Etale trivial bundles} \label{Sect finite}

Let $M$ be a compact connected complex manifold. The tangent bundle and the $i$--exterior product of the cotangent bundle of $M$ will
be denoted by $TM$ and $\Omega^i_M$ respectively. A holomorphic connection on a holomorphic vector bundle $E$ on $M$ is a first
order holomorphic differential operator $D\, :\, E\, \longrightarrow\, E\otimes\Omega^1_M$ satisfying the Leibniz
identity, which says that $D(fs)\,=\, fD(s)+ s\otimes df$ for all locally defined holomorphic section $s$ of $E$ and
all locally defined holomorphic function $f$ on $M$ (see \cite{At}). Any coherent analytic sheaf on $M$ admitting a holomorphic
connection is locally free (see \cite[p.~211, Proposition 1.7]{Bo} and \cite[p.~1037, Remark 2.1]{BDDH}).
A holomorphic connection $D$ is called integrable if its curvature
$D^2\, \in\, H^0(M,\, \text{End}(E)\otimes\Omega^2_M)$ vanishes identically. For an integrable holomorphic connection $D$
we have the monodromy homomorphism
\begin{equation}\label{e1}
\text{Mon}(D)\,\,:\,\,\pi_1(M,\, x_0)\, \longrightarrow\, \text{GL}(E \vert_{x_0})\, ,
\end{equation}
where $x_0\, \in\, M$ is any base point.

A holomorphic vector bundle $E$ on $M$ is called \'etale trivial
if there is a finite \'etale covering $\psi\, :\, \widetilde{M}\, \longrightarrow\, M$ such that
$\psi^*E$ is holomorphically trivial.

The following lemma is straightforward.

\begin{lemma}\label{lem1}
A holomorphic vector bundle $E$ on $M$ is \'etale trivial if and only if $E$ admits an integrable holomorphic connection
$D$ such the image of the monodromy homomorphism ${\rm Mon}(D)$ in \eqref{e1} is a finite subgroup of ${\rm GL}(E \vert_{x_0})$.

Such a connection is uniquely determined by the following property. If $\psi:\widetilde M\longrightarrow M$ is a finite Galois covering for which there exists an isomorphism $h:\psi^*E\longrightarrow \mathcal O_{\widetilde M}^r$, then $\psi^*D$ corresponds, under $h$, to the trivial connection on $\mathcal O_{\widetilde M}^r$.
\end{lemma}

\begin{proof}
If the image of $\text{Mon}(D)$ is finite, then consider the \'etale Galois covering
$\psi\, :\, \widetilde{M}\, \longrightarrow\, M$ corresponding to $\text{kernel}(\text{Mon}(D))\, \subset\,
\pi_1(M,\, x_0)$. The pulled back holomorphic connection $\psi^*D$ on $\psi^*E$ has trivial monodromy, and
hence $\psi^*E$ is holomorphically trivial.

Given any a connected finite \'etale covering $\psi_1\, :\, \widetilde{M}_1\, \longrightarrow\, 
M$, there is a connected finite \'etale covering $\psi_2\, :\, \widetilde{M}\, 
\longrightarrow\, \widetilde{M}_1$ such that $\psi\, :=\, \psi_1\circ\psi_2$ is Galois. If
$\psi^*_1E$ is holomorphically trivial, then $\psi^*E$ is also holomorphically trivial. So 
if $E$ is \'etale trivial, then we may assume that the trivializing connected finite \'etale 
covering $\psi$ is Galois. If $\psi^* E$ is holomorphically trivial, then the natural 
evaluation map
$$
\widetilde{M}\times H^0(\widetilde{M},\, \psi^*E)\, \longrightarrow\, \psi^*E
$$
is a holomorphic isomorphism. Using this isomorphism, the trivial holomorphic connection on $\widetilde{M}
\times H^0(\widetilde{M},\, \psi^*E)$ produces an integrable holomorphic connection $\widetilde D$ on $\psi^*E$ with trivial monodromy.
This holomorphic connection $\widetilde D$ is preserved by the natural action of $\text{Gal}(\psi)$ on $\psi^*E$, and
hence it descends to a holomorphic connection on $E$. This descended holomorphic connection on $E$ is evidently integrable
and its monodromy homomorphism has finite image.

To verify the last property, we observe the following. Let $f:E\longrightarrow F$ be an arrow between holomorphic vector bundles and $\psi\,:\,\widetilde{M}\,\longrightarrow\, M$ an \'etale covering such that $\psi^*E$ and $\psi^*F$ are trivial vector bundles. Suppose that $\partial_E$ and $\partial_F$ are integrable connections on $E$ and $F$, respectively, such that 
\[
\psi^*(f)\,\,:\,\,(\psi^*E,\,\psi^*\partial_E)\,\longrightarrow\,(\psi^*F,\,\psi^*\partial_F)
\] 
is horizontal. It then follows that $f$ is horizontal. 
\end{proof}

Take any polynomial $p(x)\,=\, \sum_{i=0}^d a_ix^i$, where $a_i$ are nonnegative integers. For any holomorphic vector bundle
$E$ on $M$, define the holomorphic vector bundle
$$
p(E)\,:=\, \bigoplus \left(E^{\otimes i}\right)^{\oplus a_i}\, ,
$$
where $E^{\otimes 0}$ is the trivial holomorphic line bundle ${\mathcal O}_M$, and $F^{\oplus 0}=0$ for any vector bundle $F$. A holomorphic vector bundle $E$ is called finite if there are two polynomials $p$ and $q$ as above, with
$p\, \not=\, q$, such that the two holomorphic vector bundles $p(E)$ and $q(E)$ are holomorphically isomorphic
\cite[p. 35]{No1}, \cite{No2}.

The following is a criterion for a vector bundle to be finite.

\begin{theorem}[{\cite[Theorem 1.1]{Bi}}]\label{thm1}
A holomorphic vector bundle $E$ on $M$ is finite if and only if it is \'etale trivial.
\end{theorem}

In \cite{No1}, \cite{No2}, Theorem \ref{thm1} was proved in the algebro-geometric set-up.
Theorem \ref{thm1} is a key ingredient in the proof of the following result (whose algebro-geometric counterpart appears in \cite[Remarks, pp. 232-3]{BdS}.)

\begin{theorem}[{\cite[Theorem 1.1]{BD}}]\label{thm2}
Let $f\, :\, X\, \longrightarrow\, M$ be a surjective holomorphic maps between compact connected complex
manifolds and $E\, \longrightarrow\, M$ a holomorphic vector bundle, such that $f^*E$ is holomorphically
trivial. Then $E$ is \'etale trivial.
\end{theorem}

\subsection{A neutral Tannakian category}

Let $M$ be a compact connected complex K\"ahler manifold. Fix a base point $x_0\, \in\, M$. In 
\cite{BDDH}, the following was studied: The category $\mathcal T(M)$ consisting of trivial vector bundles 
endowed with an integrable connection. Using the fiber functor $E\,\,\longmapsto\,\,E|_{x_0}$, 
it was deduced that $\mathcal T(M)$ is equivalent to the category of representations of an affine 
group scheme \[\Theta(M,\,\,x_0).\] A basic property of $\Theta(M,\,x_0)$ is that it is 
pro-connected \cite[Proposition 3.7]{BDDH}. 

Let ${\mathcal C}_{\mathrm{dR}}(M)$ denote the category whose objects are pairs of the form
$(E,\, D)$, where $E$ is a holomorphic vector bundle on $M$ and $D$ is an integrable
holomorphic connection on $E$. As before, morphisms from
$(E,\, D)$ to $(E',\, D')$ are all holomorphic homomorphisms of vector bundles
$h\, :\, E\, \longrightarrow\, E'$ satisfying $D'\circ h\,=\,
(h\otimes{\rm Id}_{\Omega^1_M})\circ D$ as differential operators from $E$ to
$E'\otimes\Omega^1_M$. This category is equipped with the operators of direct sum, tensor
product and dualization. More precisely, ${\mathcal C}_{\mathrm{dR}}(M)$ is a rigid abelian tensor category
(see \cite[p.~118, definition 1.14]{DMOS} for rigid abelian tensor categories).
This category ${\mathcal C}_{\mathrm{dR}}(M)$, endowed with the faithful fiber functor that sends any object
$(E,\, D)$ to the fiber $E |_{x_0}$ defines a neutral Tannakian category \cite[p.~67]{Si}. The
corresponding pro-algebraic affine group scheme over $\mathbb C$ will be denoted by $\varpi(M,\, x_0)$
(see \cite[p.~69]{Si}, \cite[Section 2]{BDDH}).
Concerning the relation between $\mathcal C_{\rm dR}(M)$ and $\mathcal T(M)$, we can say that the morphism of group schemes 
\begin{equation}\label{17.10.2022--1}
\mathbf q:\varpi(M,\,x_0)\,\longrightarrow\,\Theta(M,\,x_0)
\end{equation}
induced by the inclusion $\mathcal T(M)\longrightarrow \mathcal C_{\rm dR}(M)$ is \emph{faithfully flat}, or a {\it quotient} morphism. See Proposition 3.1 in \cite{BDHH}.

Let us now modify our point of view and enlarge the category of vector bundles on which we allow a connection. 
Let ${\mathcal T}^{\mathrm{et}}(M)$ denote the category of all pairs of the form $(E,\, D)$, where
$E$ is an \'etale trivial vector bundle on $M$ and $D$ is a holomorphic connection on $E$. This ${\mathcal T}^{\mathrm{et}}(M)$
is evidently closed under the operations of direct sum and tensor product, just as it is closed
under the operation of taking dual. For any holomorphic homomorphism between two holomorphically trivial vector bundles
over a compact connected complex manifold,
both kernel and cokernel are holomorphically trivializable because the homomorphism is given by a constant matrix. From this
it follows that for any two pairs $(E_1,\, D_1)$ and $(E_2,\, D_2)$ in ${\mathcal T}^{\mathrm{et}}(M)$,
any homomorphism $\rho\, :\, E_1\, \longrightarrow\, E_2$, satisfying the condition that
$$
D_2\circ\rho\,=\, (\rho\otimes{\rm Id}_{\Omega^1_M})\circ D_1,
$$
has the property that both $\text{kernel}(\rho)$ and $\text{cokernel}(\rho)$ are \'etale 
trivial bundles, and furthermore, $D_1$ (respectively, $D_2$) induces a holomorphic connection on 
$\text{kernel}(\rho)$ (respectively, $\text{cokernel}(\rho)$).

Therefore, ${\mathcal T}^{\mathrm{et}}(M)$ together with the faithful fiber functor on it that
sends any $(E,\, D)$ to the fiber $E\big\vert_{x_0}$ produce a neutral Tannakian category (see \cite[p.~138,
Definition 2.19]{DMOS}, \cite{Sa}, \cite[p.~76]{No2} for neutral Tannakian category).
For any neutral Tannakian category there is a naturally associated pro-algebraic affine
group scheme over $\mathbb C$ \cite[p.~130, Theorem 2.11]{DMOS} (and the remark following \cite[p.~138,
Definition 2.19]{DMOS}), \cite{Sa}, \cite[p.~77, Theorem 1.1]{No2}, \cite[p.~69]{Si}.
Consequently, the neutral Tannakian category ${\mathcal T}^{\mathrm{et}}(M)$
produces a pro-algebraic group (= affine group scheme) 
\[\mathbb{E}(M,\, x_0)\]
over $\mathbb C$ such that by means of the functor $\bullet|_{x_0}$, the category ${\mathcal T}^{\mathrm{et}}(M)$ becomes equivalent to the category of representations of $\mathbb{E}(M,\, x_0)$ on finite dimensional complex vector spaces.

The tautological functor ${\mathcal T}^{\mathrm{et}}(M) \, \longrightarrow\, {\mathcal C}_{\mathrm{dR}}(M)$ produces a homomorphism
\begin{equation}\label{e3}
\mathbf r\,\, :\,\, \varpi(M,\, x_0)\,\, \longrightarrow\, \mathbb{E}(M,\, x_0)\, .
\end{equation}

\begin{lemma}\label{lem3}
The homomorphism $\mathbf r$ in \eqref{e3} is faithfully flat (in other words, it is surjective).
\end{lemma}

\begin{proof}
This can be proved by the same method employed to verify Proposition 3.1 in \cite{BDDH}.
\end{proof}

Let $\pi^{\rm et}(M,\,x_0)$ be the pro-finite completion of the fundamental group of $M$ at 
$x_0$; it carries a natural structure of pro-algebraic group, or affine group scheme, once 
regarded as a projective limit of finite groups. This is commonly known as the {\it \'etale 
fundamental group scheme}.

\begin{lemma}\label{lem2}
There is a natural faithfully flat morphism (same as surjective morphism)
$$
\gamma\,\,:\,\, \mathbb{E}(M,\, x_0)\, \longrightarrow\, \pi^{\rm et}(M,\, x_0)
$$
to the \'etale fundamental group scheme $\pi^{\rm et}(M,\, x_0)$.
\end{lemma}

\begin{proof}
The \'etale fundamental group $\pi^{\rm et}(M,\, x_0)$ corresponds to the neutral Tannakian category defined by the \'etale
trivial
holomorphic vector bundles on $M$. In the proof of Lemma \ref{lem1} it was shown that any \'etale trivial vector
bundle has a canonical integrable holomorphic connection whose monodromy is finite. Any homomorphism between
two \'etale trivial holomorphic vector bundles is flat with respect to this canonical connection. 
Indeed, write $\mathrm d_E$
for the canonical connection on an \'etale trivial bundle $E$. Then, if $\psi: \widetilde{M}\longrightarrow M$ is a finite \'etale covering such that $\psi^*E\simeq\mathcal O_{\widetilde M}^r$, then the pulled-back connection $\psi^*\mathrm d_E$ corresponds to the trivial connection on $\mathcal O_{\widetilde M}^r$. (Note that this notion is independent of the isomorphism $\psi^*E\simeq\mathcal O_{\widetilde M}^r$.)
This being so, if $f\,:\,E\, \longrightarrow\, F$ is an arrow between vector bundles which become trivial on the \'etale covering $\psi:\widetilde{M}\longrightarrow M$, then the arrow $\psi^*f$ is also horizontal with respect to the connections $\psi^*(\mathrm d_E)$ and $\psi^*(\mathrm d_F)$. Consequently, surjectivity of $\psi$ assures that $h$ is horizontal. 

The lemma
is a straightforward consequence of these.
\end{proof}

The group of multiplicative characters $G\, \longrightarrow\, {\mathbb G}_m=\mathbb C^*$ of an 
affine group scheme $G$ will be denoted by ${\mathbb X}(G)$.

Let
\begin{equation}\label{e2}
\gamma^*\,\, :\,\, {\mathbb X}(\pi^{\rm et}(M,\, x_0))\,\, \longrightarrow\,\, {\mathbb X}(\mathbb{E}(M,\, x_0))
\end{equation}
be the homomorphism induced by $\gamma$ in Lemma \ref{lem2}.

\begin{proposition}\label{prop1}
The image of the homomorphism $\gamma$ in \eqref{e2} is the subgroup consisting of all characters of finite order.
\end{proposition}

\begin{proof}
Any character in ${\mathbb X}(\pi^{\rm et}(M,\, x_0))$ is of finite order, and hence the image of $\gamma^*$ is
contained in the subgroup of ${\mathbb X}(\mathbb{E}(M,\, x_0))$ consisting of all characters of finite order.

Let $L$ be an \'etale trivial line bundle on $M$ and $D$ a holomorphic connection on $L$. If the curvature of
$D$ is $R\, \in\, H^0(M,\, \Omega^2_M)$, then the curvature of the connection on $L^{\otimes n}$ induced by
$D$ is $n\cdot R$. Therefore, if $(L,\, D)$ corresponds to a multiplicative character of $\mathbb{E}(M,\, x_0)$
of finite order, then $D$ is integrable.

To prove the proposition we need to show the following: Let $D$ be an integrable holomorphic connection on
an \'etale trivial line bundle $L$ on $M$. If the monodromy homomorphism for $D$ has finite image, then $D$ is the
canonical connection on $L$ constructed in Lemma \ref{lem1}.

Since the monodromy homomorphism for $D$ has finite image, there is a connected \'etale Galois covering
$$
\psi\,\,:\,\, \widetilde{M}\,\,\longrightarrow\,\, M
$$
such that $\psi^*L$ is the trivial holomorphic line bundle and $\psi^*D$ is the trivial connection
on $\psi^*L$. From the construction of the connection in the proof of Lemma \ref{lem1} it follows immediately
that $D$ coincides with the canonical connection on $L$ in Lemma \ref{lem1}. This completes the proof.
\end{proof}

Let us end this section by rendering explicit the group of {\it additive characters} of $\mathbb E(M,\,x_0)$; this is a key observation in \cite{BDDH}. The natural functor $\mathcal T(M)\longrightarrow\mathcal T^{\rm et}(M)$ defines a morphism of group schemes 
\[\mathbf{p}\,\,:\,\,
 \mathbb E(M,\,x_0)\,\longrightarrow\,\Theta(M,\,x_0)
 \]
 which must be automatically surjective in view of the fact that
$\mathbf q\,:\,\varpi(M,\,x_0)\,\longrightarrow\,\Theta(M,\,x_0)$ is. For an affine group scheme, let $\mathbb X_a(G)$ stand for the group of homomorphisms $G\,\longrightarrow\,\mathbb G_a$. We then obtain an injection 
 \begin{equation}\label{18.10.2022--2}
 \mathbb X_a(\Theta(M,\,x_0))\,\longrightarrow\,{\mathbb X}_a({\mathbb E}(M,\,x_0))
 \end{equation}
 
\begin{lemma}\label{19.10.2022--1}The arrow in eq. \eqref{18.10.2022--2} is a bijection. In particular, the group of additive characters of $\mathbb E(M,\,x_0)$ is isomorphic
to $H^0(M,\,\Omega_M^1)$. 
\end{lemma}

\begin{proof}Let $G$ be an affine group scheme. We first note that $\mathbb X_a(G)$ can be canonically identified with the group $\mathrm{Ext}^1_G( \mathbb C, \mathbb C)$ of extensions of the trivial representation by itself. On the side of vector bundles, we are then required to show that given an extension 
\[0\,\longrightarrow\, (\mathcal O_M,\, \mathrm d)\,\longrightarrow \,(E,\,D)
\,\longrightarrow\, (\mathcal O_M,\,\mathrm d)\,\longrightarrow\,0,\]
where $E\in\mathcal T^{\rm et}(M)$, we must immediately have $E\,\simeq\, \mathcal O_M^2$. Let
$\psi\,:\,{\widetilde M}\,\longrightarrow\, M$ be the an \'etale Galois covering with group $G$ trivializing $E$ and let $y_0\in \psi^{-1}(x_0)$.
Giving $E$ the canonical connection (see Lemma \ref{lem1}), call it $\mathrm d_E$, the sequence 
\[1\,\longrightarrow\,(\mathcal O_M,\,\mathrm d)\,\longrightarrow\, (E,\,{\rm d}_E)\,
\longrightarrow\,(\mathcal O_M,\,\mathrm d)\,\longrightarrow\,0\]
is also exact. Since $\psi^*E$ is trivial, it is the case that
$$\pi_1({\widetilde M},\,y_0)\,\longrightarrow\,
 \mathrm{GL}(E|_{x_0})
\,=\,\mathrm{GL}_2(\mathbb C)$$ is trivial and hence $\pi_1(M,\,x_0)
\,\longrightarrow\, \mathrm{GL}_2(\mathbb C)$ factors through the finite quotient $G$. Now $\mathrm{Mon} ({\rm d}_E)$ has image in the subgroup of strict upper triangular matrices, which is isomorphic to $\mathbb C$. Since $\mathbb C$ does not have any finite subgroup other than the trivial, we conclude that $\mathrm{Mon} (\mathrm{d}_E)$ is trivial. Hence, $E$ is holomorphically trivial. 

The last claim follows from \cite[Lemma 3.5]{BDDH}. 
\end{proof}
 
\subsection{The special case of Riemann surfaces}

In this subsection assume that $\dim_{\mathbb C} M\,=\, 1$. Since $\Omega^2_M\,=\, 0$, any 
holomorphic connection on $M$ is integrable,

Let ${\mathcal C}_{B}(M)$ denote the category whose objects are all finite dimensional complex representations
of the topological fundamental group $\pi_1(M,\, x_0)$. Sending a holomorphic connection $(E,\, D)$ to its
monodromy homomorphism $\pi_1(M,\, x_0)\, \longrightarrow\, \text{GL}(E \vert_{x_0})$ we obtain an equivalence
between the two categories ${\mathcal C}_{\mathrm{dR}}(M)$ and ${\mathcal C}_{B}(M)$.
Using the tautological fiber functor ${\mathcal C}_{B}(M)$ defines a neutral Tannakian category. In view of
the above equivalence of categories, the
pro-algebraic affine group scheme over $\mathbb C$ corresponding to ${\mathcal C}_B(X)$
is the group scheme $\varpi(X,\, x_0)$ \cite[p.~69, Lemma 6.1]{Si}. In other words, $\varpi(M,\, x_0)$ is the
pro-algebraic completion of $\pi_1(M,\, x_0)$.

Let $N$ be a compact connected Riemann surface and 
\begin{equation}\label{e4}
\phi\,:\, M\, \longrightarrow\, N
\end{equation}
an orientation preserving map. The corresponding homomorphism of
fundamental groups
$$\phi_*\,:\, \pi_1(M,\, x_0)\, \longrightarrow\, \pi_1(N,\, \phi(x_0))$$
produces a homomorphism
\begin{equation}\label{e5}
\phi_*\,:\, \varpi(M,\, x_0)\, \longrightarrow\, \varpi(N,\, \phi(x_0))
\end{equation}
between the pro-algebraic completions.

\begin{proposition}\label{prop2}
Take any $\phi$ as in \eqref{e4} which is a homeomorphism. If the homomorphism $\phi_*$ is \eqref{e5}
descends to a homomorphism
$$
\mathbb{E}(M,\, x_0)\,\, \longrightarrow\, \, \mathbb{E}(N,\, \phi(x_0))\, .
$$
of quotients (see Lemma \ref{lem3}), then the two Riemann surfaces $M$ and $N$ are isomorphic.
\end{proposition}

\begin{proof}
Its proof is very similar to the proof of Theorem 4.1 of \cite[p.~1050]{BDDH}; the key point is the existence of an arrow
$\mathbb{E}(M,\, x_0)\,\, \longrightarrow\, \, \mathbb{E}(N,\, \phi(x_0))$ which then assures, through Lemma \ref{19.10.2022--1}, the existence of an arrow $H^0(N,
\, \Omega_N^1)\,\longrightarrow\, H^0(M,\,\Omega_M^1)$ rendering 
\[
\xymatrix{H^1_{\rm dR}(N,\,\mathbb C)\ar[rr]^{\phi^*}&&H^1_{\rm dR}(N,\,\mathbb C)
\\
H^0(N,\,\Omega_N^1)\ar[u]\ar[rr]&&H^0(M,\,\Omega_M^1)\ar[u]}
\]
commutative. 
\end{proof}

\section{Holonomy of connection on trivial bundle} \label{Sect Trivial bundle}

Let $M$ be a compact connected complex manifold and $G$ a complex connected Lie group. The
Lie algebra of $G$ will be denoted by $\mathfrak g$. Consider 
the trivial holomorphic principal $G$--bundle
$$E\, :=\, M\times G \, \longrightarrow\, M$$
on $M$. Take a holomorphic connection $D$ on $E$. Denoting the trivial connection on $E$ by $D_0$, we have
$$
D-D_0\,\,\in\,\, H^0(M,\, \Omega^1_M\otimes_{\mathbb C}{\mathfrak g})\,=\, 
H^0(M,\, \Omega^1_M)\otimes_{\mathbb C}{\mathfrak g}\, .
$$
Let
$$
{\mathbb D}\,\,:\,\, H^0(M,\, \Omega^1_M)^*\,\, \longrightarrow\,\,\mathfrak g
$$
be the homomorphism given by $D-D_0$. The complex Lie subalgebra of $\mathfrak g$ generated by the image of
the homomorphism $\mathbb D$ will be denoted by ${\mathfrak g}_D$. Let $G_D$ denote the unique smallest
connected closed complex Lie subgroup of $G$ such that the Lie algebra of $G_D$ contains ${\mathfrak g}_D$.

Consider the holomorphic reduction of structure group of $E$
\begin{equation}\label{c1}
E_D\,:=\, M\times G_D \, \subset\, M\times G \,=\, E
\end{equation}
to the subgroup $G_D$.

\begin{proposition}\label{prop-h}
The connection $D$ preserves $E_D$ in \eqref{c1} (equivalently, $D$ is induced by a 
connection on $E_D$).

If $E_H\, \subset\, E$ is a holomorphic reduction of the structure group of $E$ to a 
complex Lie subgroup $H\, \subset\, G$ such that
\begin{itemize}
\item the connection $D$ preserves $E_H$, and

\item there is a point $x_0\, \in\, M$ such that $(x_0,\, e)\, \in\, E_H$, where $e\, \in\, G$ is the
identity element,
\end{itemize}
then $G_D\, \subset\, H$ and $E_D\, \subset\, E_H$.
\end{proposition}

\begin{proof}
Since ${\mathbb D}(H^0(M,\, \Omega^1_M)^*))\, \subset\, {\mathfrak g}_D\, \subset\, \text{Lie}(G_D)$, it
follows that the connection $D$ preserves the reduction $E_D$ in \eqref{c1}. The second statement is
equally straightforward.
\end{proof}

\section{Expressing $\varpi\longrightarrow\Theta$ with the aid of iterated 
integrals}\label{Sect iterated}

Let $M$ be a complex compact K\"ahler manifold and $x_0$ a point of it. Recall that $\varpi\,=\,\varpi(M,\,x_0)$ (respectively,
$\Theta\,=\,\Theta(M,\,x_0)$) 
is the Tannakian group associated to the category of integrable holomorphic connections ${\mathcal C}_{\mathrm{dR}}(M)$ (respectively,
integrable holomorphic connections on trivial vector bundles ${\mathcal T}(M)$) over $M$. And our 
objective in this section is to render the arrow $\mathbf q$ of eq. \eqref{17.10.2022--1} more explicit by using {\it iterated integrals}, therefore linking it to other works like 
\cite{parshin} and \cite{hain86}.

It is showed in \cite{biswas-hai-dos_santos22} that $\Theta$ is the Tannakian envelope of a co-commutative Hopf algebra ${\mathfrak A}$ 
constructed via the free algebra on the vector space $H^0(M,\, \Omega_M^1)^*$. Constructing ${\mathfrak A}$ is quite simple: letting 
$\{\beta_i\}$ be a basis of $H^0(M,\, \Omega_M^2)$, define \[B_i\,\in\, H^0(M,\, \Omega_M^1)^*\otimes 
H^0(M,\, \Omega_M^1)^*,\quad\omega\otimes\omega'\,\longmapsto\,\beta_i(\omega\wedge\omega').\] Then ${\mathfrak A}\,=\,\mathbb 
T(H^0(M,\, \Omega_M^1)^*)/(B_i)$. See \cite[Section 2]{biswas-hai-dos_santos22} for details about ${\mathfrak A}$. Concerning the relation 
between $\Theta$ and ${\mathfrak A}$, we can say, in a nutshell, that the restricted dual of ${\mathfrak A}$ \cite[VI]{sweedler69}, 
\[{\mathfrak A}^\circ\,=\,\varinjlim ({\mathfrak A}/I)^* ,\quad \text{$I$ an ideal of finite codimension},\] is the $\mathbb{C}$--algebra of 
an affine group scheme isomorphic to $\Theta$. In fact, we have natural tensor equivalences
\[
\comodules{{\mathfrak A}^\circ}\,\,\stackrel{\sim}{\longrightarrow}\,\,
\modules{\mathfrak A}\,\,\stackrel{\mathcal{V}}{\longrightarrow}\,\,
{\mathcal T}(M),
\]
where $\mathcal V$ is constructed as follows (see \cite[Section 3]{biswas-hai-dos_santos22}
for more). Given $$f\,:\,{\mathfrak A} \,\longrightarrow\, \mathrm{End}(E)$$ an object of
$\modules{\mathfrak A}$, we associate to it the connection $\mathrm d_f$ on $\mathcal 
O_M\otimes E$ defined by the $\mathrm {End}(E)$--valued form derived from the composition
\begin{equation}\label{16.09.2022--1}
H^0(M,\, \Omega_M^1)^*\,\longrightarrow\, {\mathfrak A}\,\longrightarrow\,\mathrm{End}(E).
\end{equation}
Under this equivalence, the forgetful functor is naturally 
isomorphic to $\bullet|_{x_0}$.

Now, ${\mathcal C}_{\mathrm{dR}}(M)$ is also equivalent to the category of representations of the topological fundamental group 
$\Gamma\,=\,\pi_1(M,\,x_0)$ by means of the monodromy, and the latter category can be shown to be isomorphic to the category of ``continuous'' 
modules over a certain topological $\mathbb C$--algebra, $\mathbb C\Gamma\,\widehat{\,\,}$, associated to $\mathbb C\Gamma$. Thus, in 
order to render $\mathbf q$ explicit, we shall elaborate on the expression of the monodromy representation for objects in ${\mathcal 
T}(M)$.

\subsection{Parallel transport and monodromy}

Let $\Gamma$ be the fundamental group of $M$ based at $x_0$. Given $(\mathcal E,\,\nabla)\,\in\,{\mathcal C}_{\mathrm{dR}}(M)$, 
the monodromy representation
\[
\mathrm{Mon}^{x_0}_\nabla\,:\,\Gamma\,\longrightarrow\, \mathrm{GL}(\mathcal E|_{x_0})
\]
is obtained via parallel transport. Write $E$ for the (geometric) holomorphic vector bundle whose sheaf of sections is $\mathcal{E}$. 
Given a smooth loop $\gamma\,:\,[0,\,1]\,\longrightarrow\, M$ based at $x_0$, let
\[
P_{\gamma}\,:\,\mathcal{E}|_{\gamma(0)}\,\stackrel{\sim}{\longrightarrow}\,\mathcal{E}|_{\gamma(1)}
\]
be the isomorphism of complex vector spaces obtained from parallel transport
$$P_{\gamma}\,:\,E|_{\gamma(0)}
\,\longrightarrow\, E|_{\gamma(1)}$$ via the canonical identification $\mathcal{E}|_{x}
\,\stackrel{\sim}{\longrightarrow}\, E|_{x}$. In other words, 
we extend $\nabla$ to a $C^\infty$ connection, find a $C^\infty$ lift
$\widetilde\gamma\,:\,[0,\,1]\,\longrightarrow\, E$
of $\gamma$ such that $\widetilde{\gamma}(0)\,=\,e$ and $\nabla_{\gamma'}\widetilde\gamma \,=\,0$, and then define 
\[
P_\gamma(e)\,=\,\widetilde{\gamma}(1).
\] 
As the composition $\gamma_1*\gamma_2$ in $\Gamma$ is the path $\gamma_1$ followed by $\gamma_2$, we must write 
\[\mathrm{Mon}_\nabla^{x_0}(\gamma)\,=\,P_{\gamma^{-1}}.\]

\subsection{Parallel transport, monodromy in $\mathcal T(M)$ and iterated integrals}

In this section we explain how to write the monodromy representation in terms of iterated 
integrals. This is explained in \cite[\S~2]{hain86}, and attributed to Chen, but Hain 
adopts different conventions (like letting parallel transport ``act on the right'' and 
working with connections defined by the negative of a form, see (2.1) and line $-11$ on 
p. 614 of \cite{hain86}), so that it is worth explaining it in full here. Let $E$ be a 
vector space of finite dimension and
\[
\omega\,\,\in\,\, H^0(M,\, \Omega^1_M)\otimes \mathrm{End}(E)
\]
a differential form with values on $\mathrm{End}(E)$. 
Write ${\rm d}_{\omega}$ for the connection on the trivial vector bundle $M\times E\,\longrightarrow\, M$ whose associated 1--form is $\omega$. 
By definition, if $s\,=\,(\mathrm{id},\,e)$ is a section on some open set then, for any $C^\infty$ vector field $v$,
\[
({\rm d}_\omega)_v(s)\,\, =\,\, (\mathrm{id} ,\, v(e)+\omega(v)\cdot e).
\]
Given a smooth loop $\gamma\,:\,[0,\,1]\,\longrightarrow\, M$
about $x_0$, a lift $t\,\longmapsto\,(\gamma(t),\,e(t))$ is parallel if and only if
\[
e'(t)+A(t)e(t)\,=\,0,
\]
where $A(t)\,=\,\omega(\gamma'(t))$. It then follows that 
the parallel transport of $(x_0,\,e_0)\,\in\, (M\times E)|_{x_0}$ along $\gamma$, call it 
\[
P_{\gamma}(x_0,\,e_0)
\]
is the value $y(1)$ of a solution to 
\[
\begin{array}{lcl}y'(t)&=&-A(t)y(t)
\\
y(0)&=&e_0.
\end{array}
\]
Solutions to this sort of equation are described concisely using Picard iteration. 

Recall the following fundamental result. (A proof can be adapted from the proofs of the Picard-Lindel\"of Theorem, which is to be found 
in almost all texts on ordinary differential equations, e.g. \cite{coddington61}, Chapter 6, Theorem 4 and Theorem 6.)

\begin{theorem}[Picard-Lindel\"of]\label{24.08.2020--1}Let $R$ be a finite dimensional associative $\mathbb{C}$--algebra and $F\,
:\,[0,\,1]\,\longrightarrow\, R$ be a continuous map. Given $f_0\,\in\, R$, define inductively continuous functions $[0,\,1]
\,\longrightarrow\, R$ by 
\[
f_{k+1}(t)\,=\,f_0+\int_0^tF(s)\cdot f_k(s)\,\mathrm ds.
\]
\begin{enumerate}[(1)]\item The limit $f(t)\,=\,\lim_kf_k(t)$ exists for each $t\,\in\,[0,\,1]$ and $f'\,=\,F \cdot f$. 
\item Let $\delta_k\,=\,f_k-f_{k-1}$, for each $k\,\ge\,1$. Then, $\delta_1(t)\,=\,\int_0^tF(s)\cdot f_0\,\mathrm ds$,
 \[
\delta_{k+1}(t)\,=\,\int_0^tF(s)\delta_k(s)\,\mathrm ds,
\]
for $k\,\ge\,1$
and
\[
f(t)\,\,= \,\,f_0+ \sum_{k=1}^\infty\delta_k(t).
\]
\end{enumerate}
\end{theorem}

The functions $\{\delta_k\}$ appearing in the above statement have a well-known expression in
terms of iterated integrals. Let $R$ be as in Theorem \ref{24.08.2020--1} and consider
$R$--valued $C^\infty$ differential forms $\rho_1,\,\cdots,\,\rho_m$. Given a smooth path $c
\,:\,[0,\,1]\,\longrightarrow\, M$, let $F_i\,=\,\rho_i(c')$ and define 
\[
\underbrace{G_1(t)\,=\,\int_0^tF_1(s)\,\mathrm ds}_{\in R}, \quad G_{i+1}(t)=\underbrace{\int_0^t G_i(s)\cdot F_{i+1}(s)\,\mathrm ds}_{\in R}
\]
and
\[\underbrace{
\int_c\rho_1\cdots\rho_m\,=\,G_m(1)}_{\in R}
\] 
(See \cite[p.~359]{chen71} and \cite[\S~1]{hain86}, among others.)
For example, 
\[
\int_c\rho_1\rho_2\,\,=\,\, \int_0^1\left(\int_0^{s_2}F_1(s_1)\,\mathrm ds_1\right)\cdot F_2(s_2)\,\mathrm ds_2.
\]
Alternatively, it is sometimes convenient to think of the iterated integral as an integral over a simplex:
\[
\int_c\rho_1 \cdots\rho_m\,=\, \int_{0\le s_1\le\cdots\le s_n\le1} F_1(s_1)\cdots F_m(s_m)\,\mathrm ds_1\cdots \mathrm ds_m.
\]

Returning to the problem of expressing the parallel transport of $$(x_0,\,e_0)\,\in\,(M\times 
E)|_{x_0}$$ along the path $\gamma$ with respect to $\mathrm{d}_\omega$, the above terminology 
and Theorem \ref{24.08.2020--1} allow us to write
\[
P_{\gamma}(x_0,\,e_0)\,=\,\left(x_0\,\,,\,\, e_0-\int_0^1 A(s_1)\,\mathrm ds_1\cdot e_0+\int_{0\le s_1\le s_2\le 1} A(s_1) A(s_2) \,\mathrm ds_1 \,\mathrm ds_2\cdot e_0-\cdots\right).
\]
A simple change of variables shows that 
\[
\int_{\gamma^{-1}}\omega_1\cdots\omega_m\,\,=\,\,(-1)^m\int_{\gamma}\omega_m\cdots\omega_1
\]
and hence 
\begin{equation}\label{27.08.2020--1}
\mathrm{Mon}^{x_0}_{{\rm d}_\omega}(\gamma)\,:\,(x_0,\,e_0)\,\longmapsto\,
\left(x_0\,\,,\,\,e_0+\int_{\gamma} \omega\cdot e_0 +\int_{\gamma}\omega^2\cdot e_0+\int_{\gamma }\omega^3\cdot e_0+\cdots\right).
\end{equation}(Here, obviously, $\omega^n$ stands for $\underbrace{\omega\cdots\omega}_{n}$.)

\subsection{Algebraic interlude}\label{20.09.2022--2}

To sustain the arguments concerning the relation between ${\mathfrak A}$ and $\mathbb C\Gamma$ 
necessary in explaining the morphism $\mathbf q$, we shall require some material on 
completions of rings and Hopf duals. We fix a field $K$.

\subsubsection*{Completions} Let $A$ be an associative and unital $K$--algebra. Write 
\[
S_{A}\,=\,
\left\{\begin{array}{c}\text{two sided ideals of} \\ \text{finite codimension in $A$}\end{array}\right\};
\] this set it partially ordered by decreeing $I\,\le\, J$ if $I\,\supset\, J$. 
Consider the subset of $2^A$ defined by 
\[
\{a+I\}_{a\in A,\,I\in S_A}
\]
and let $\mathfrak O$ be the unique topology on $A$ having this set as a
sub-base \cite[p.~47--48]{kelley55}. Of course, $S_A$ is local base at $0$ \cite[p.~50]{kelley55} and all ideals of $S_A$ are open, hence closed. 

Put 
\[
\widehat A\,\,=\,\,\varprojlim_{I\in S_A}A/I; 
\]
endowing each $A/I$ with the \emph{discrete topology} and $\widehat A$ with the projective limit topology, $\widehat A$ becomes a 
topological ring. (Rings as $\widehat A$ are called pseudo-compact \cite[VIIB, \S~0]{SGA3}.) Let $\iota\,:\,A\,\longrightarrow\, \widehat 
A$ be the evident map; according to \cite[III.7.3, Cor.~1]{bourbaki-general-topology}, $\iota$ is a completion of $A$ for $\mathfrak O$.

Given $\mathfrak{a}\,\in\, S_A$, let $p_{\mathfrak a}\,:\,\widehat A\,\longrightarrow\, A/\mathfrak{a}$ be the canonical map and denote by $\widehat {\mathfrak 
a}$ the open and closed ideal ${\rm Ker}(p_{\mathfrak a})$. Note that $A/\mathfrak a\,\stackrel{\sim}{\longrightarrow}\,
\widehat A/\widehat{\mathfrak a}$.

\begin{lemma} The family $\{\widehat{\mathfrak a}\}_{\mathfrak{a}\in S_A}$ is a local base at $0$. The ideal $\widehat{\mathfrak a}$ is the closure of $\iota\mathfrak{a}$ in $\widehat A$.
\end{lemma}

\begin{proof} We only prove the second statement. Clearly $\iota(\mathfrak{a})\,\subset\,
\widehat {\mathfrak a}$. Let $$x\,=\,(x_I)\,\in\,\widehat {\mathfrak a}.$$ 
Let $U$ be an open neighbourhood of $x$ and let $I_1,\,\cdots,\,I_r\,\in\, S_A$ be such that
$$V\,=\,\{(y_I)\,\in\,\widehat A\,:\,y_{I_j}\,=\,x_{I_j},\,\,\forall\,j\}$$ is contained in $U$.
Let $J\,=\,I_1\cap\cdots\cap I_r\cap\mathfrak a$; it clearly belongs to $S_A$. Then, if $\xi\,\in\, A$ is such that
$\xi+J\,=\,x_J$, we conclude that $\iota\xi \,\in\, V\subset U$ and $\xi \,\in\,{\mathfrak a}$.
\end{proof}

\begin{example}
Suppose that $K$ is algebraically closed and $A$ is a commutative $K$--algebra of finite type. 
Then $\widehat A$ is $\prod_{\mathfrak m}\widehat{A}_{\mathfrak m}$, where $\mathfrak m$ ranges 
over the maximal ideals.
\end{example}

Let $\modules A$ be the category of $A$--modules which have finite dimension as $K$--spaces. 
Analogously, let $\modulesc{\widehat A}$ be the category of \emph{continuous} $\widehat 
A$--modules, i.e., those (left) $\widehat A$--modules of finite dimension over $K$ which are 
annihilated by some $\widehat {\mathfrak a}$ with $\mathfrak{a}\,\in\, S_A$. (Note that, although 
any $\widehat A$--module of finite dimension is annihilated by an ideal $\mathfrak I\,\in\, 
S_{\widehat A}$, we cannot assure a priori that such an ideal contains a certain $\widehat I$.)
Using the natural morphism $A\,\longrightarrow\, \widehat A$, we obtain a functor
\[
\modulesc {\widehat A} \,\longrightarrow\, \modules A 
\] 
which is easily seen to be an \emph{isomorphism of categories}. Indeed, for each $E\,\in\,
\modules A$, the ideal $\mathfrak{a}\,=\,\mathrm{Ann}(E)$ belongs to $S_A$, and hence we
can define a structure of $\widehat A$--module by
\[
(x_I)\cdot e \,=\, x_{\mathfrak a}\cdot e,\quad (x_I)\in\widehat A,\,\,e\,\in\, E.
\]

\subsubsection*{The restricted dual}

Let $A$ be an associative and unital $K$--algebra and define, as in \cite[Ch. VI]{sweedler69}, 
the restricted dual
\[\begin{split}
A^\circ &\,=\, \varinjlim_{I\in S_A}(A/I)^*
\\&\,=\,\{\alpha\in A^*\,:\,\text{$\alpha$ vanishes on some $I\in S_A$}\}. 
\end{split}\]
This is a co-algebra \cite[6.0.2, p.~113]{sweedler69} and 
\[
\widehat A\,=\,(A^\circ)^*
\]
since dualization takes direct limits into projective limits. Letting $\comodules {A^\circ}$ be 
the category of (right) comodules over $A^\circ$ having finite dimension as $K$--spaces, we have 
functors
\begin{equation}\label{20.09.2022--1}
\sigma_A\,:\,\comodules {A^\circ}\,\longrightarrow\, \modules A\quad\text{and}\quad\tau_A
\,:\,\modules A\,\longrightarrow\, \comodules {A^\circ}
\end{equation}
which together define an {\it isomorphism of categories}, cf. \cite[Theorem 2.1.3]{sweedler69}
\cite[1.6.4, p.~11]{montgomery93}. For future application, note that if $E\,\in\,
\comodules {A^\circ}$ has basis $\{e_i\}_{i=1}^m$ and co-multiplication defined by
$(a^\circ_{ij})\,\in\,\mathrm{Mat}_m((A/I)^*)$, 
then 
\[
a e_j \,\,= \,\,\sum_i a^\circ_{ij}(a)e_i.
\] Conversely, if $F\,\in\,\modules A$ has basis $\{f_i\}_{i=1}^n$, is annihilated by
$\mathfrak a$, and the action is given by 
\[
af_j\,\,=\,\,\sum_i\rho_{ij}(a)f_i,\]
then the co-module structure of $F$ is defined by the matrix $(\rho_{ij})\,\in\,\mathrm{Mat}_n( 
(A/\mathfrak a)^*)$.

Note that $(A^\circ)^*$ is now an algebra \cite[Proposition 1.1.1, p.~9]{sweedler69} which is 
easily seen to be the completion $\widehat A$ studied above.

Let now $B$ be another algebra and 
\[
v\,:\,A^\circ\,\,\longrightarrow\,\, B^\circ
\]
an arrow of co-algebras. Taking duals, we obtain an arrow of $K$--algebras $u\,:\,\widehat B \,
\longrightarrow\, \widehat A $. Alternatively, $u$ can be described as follows. For each $I\,
\in\, S_A$, there exists $J\,\in\, S_B$ such that $v ( (A/I)^*)\,\subset\,(B/J)^*$.
In addition, let $I_v\,\in\, S_B$ be the largest ideal with this property. We arrive at an order
preserving map $(-)_v\,:\,S_A\,\longrightarrow\, S_B$ and to an arrow of $K$--spaces 
\[
u_I\,:\,B/I_v\,\longrightarrow\, A/I
\]
whose transpose is $v|_{(A/I)^*}$. Note that {\it $u_I$ is a morphism of algebras} since
$v|_{(A/I)^*}$ is a morphism of co-algebras. 
A moment's thought shows that for $J\,\subset\, I$ in $S_A$, the diagram
\[
\xymatrix{
B/I_v\ar[r]^{u_I} & A/I
\\
\ar[u]^{{\rm canonic}}B/J_v\ar[r]_{u_J} & A/J\ar[u]_{\rm canonic}
}
\]
commutes and we obtain an arrow of $K$--algebras
\begin{equation}\label{17.09.2022--1}
u\,\,: \,\,\varprojlim_{J\in (S_A)_{v}}B/J\,\longrightarrow\, \varprojlim_{I\in S_A} A/I. 
\end{equation}
For each $I\,\in\, S_A$ (respectively, $J\,\in\, S_B$), endow $A/I$, (respectively, $B/J$) with
the \emph{discrete} topology and give the above projective limits the product topology. This
being so, general topology assures that $u$ is continuous. 

We then obtain from eq. \eqref{17.09.2022--1} a morphism of topological $K$--algebras 
\[
u\,:\,\widehat B \,\longrightarrow \,\widehat A
\]
and consequently a functor 
\[u^\#\,:\,
\modulesc{\widehat A}\,\longrightarrow\,\modulesc{\widehat B}. 
\]
For the sake of communication, the morphism $u$ shall be called the \emph{continuous transpose of $v$}. 
\begin{lemma}\label{20.09.2022--3}The following diagram of categories 
\[
\xymatrix{
\modulesc {\widehat A}\ar[d]_\sim\ar[rr]^{u^\#} && \modulesc{\widehat B} \ar[d]^\sim
\\
\ar[d]_\sim\modules A&& \modules B\ar[d]^\sim
\\
\comodules{A^\circ}\ar[rr]_{v^\#}&&\comodules{B^\circ}
}
\]
is strictly commutative. In particular, given $E\,\in\,\comodules{A^\circ}$, $e\,\in
\,E$ and $b\,\in\,
\widehat B$, we have \begin{equation}\label{23.09.2022--1}b\cdot e \,=\, u(b)\cdot e,
\end{equation} where on the left-hand-side of this equation, $\widehat B$ acts by means of $\modulesc {\widehat B}
\,\stackrel{\sim}{\longrightarrow}\,
\comodules{B^\circ}$, and on the right-hand-side, $u(b)$ it acts by means of $\modulesc{\widehat A}\,\stackrel{\sim}{\longrightarrow}\,
\comodules{A^\circ}$. 
\end{lemma}

\begin{proof}
We require the constructions made in the beginning of the section. Let $E\,\in\, 
\modulesc{\widehat A}$ and pick a basis $\{e_i\}$. Let $\mathfrak a\,\in\, S_A$ be such that 
$\widehat{\mathfrak a}$ annihilates $E$. Now, the coefficients $\{\varphi_{ij}\}$ determining 
the $A^\circ$--comodule structure of $E$ with respect to the basis $\{e_i\}$ are determined by
\begin{equation}\label{24.09.2022--1}
ae_j\,=\,\sum_i\varphi_{ij}(a)\cdot e_i.
\end{equation}
(Note that $\varphi_{ij}\,\in\,(A/\mathfrak a)^*\,\subset\, A^\circ$.) Consequently, the
coefficients of the $B^\circ$--comodule $E$ coming from the composition via $v^\#$ are
$v(\varphi_{ij})$. By construction, $v(\varphi_{ij})\,\in\, (B/\mathfrak a_v)^*$. 

We now examine the structure of $B^\circ$--comodule on $E$ obtained through the composition via
$u^\#$. Clearly, $\widehat{\mathfrak a}_v\,\subset\,\widehat B$ annihilates $E$. For each
$\alpha\,\in\, (A/\mathfrak a)^*$ and $b\,\in\, B/\mathfrak a_v$, we have, as $u_{\mathfrak a}
\,=\,v^*$, that 
\[
\alpha\circ u_{\mathfrak a}(b)\,=\,v(\alpha)(b). 
\]
This shows that $v(\varphi_{ij})\,=\,\varphi_{ij}\circ u_{\mathfrak a}$. By definition of $u^\#$ 
and \eqref{24.09.2022--1} we have $$be_j\,=\,u_{\mathfrak a}(b)e_j\,=\,\sum_i\varphi_{ij}(u_{\mathfrak 
a}(b))e_i$$ and the coefficients of the $B^\circ$--module $E$ are $\varphi_{ij}\circ u_{\mathfrak 
a}$. This shows commutativity of the diagram on the level of object. The case of arrows being 
trivially verified, we conclude our proof.
\end{proof}

\subsection{Fiber functor and the Tannakian picture}\label{20.09.2022--4}

Using the inclusion $\mathcal T(M)\,\longrightarrow \, \mathcal C_{\mathrm{dR}}(M)$ 
and Tannakian theory, we shall construct an arrow of Hopf algebras ${\mathfrak 
A}^\circ\,\longrightarrow \,(\mathbb C\Gamma)^\circ$ and from it define, using Section 
\ref{20.09.2022--2}, an arrow of $\mathbb C$--algebras $u\,:\,{\mathbb C}\Gamma\,\widehat{\,\,}
\,\longrightarrow \, \widehat{\mathfrak A}$. In what follows, let
us denote by an upper hashtag the functor between categories of modules or comodules obtained from
a morphism of algebras or co-algebras. 

Define a tensor functor 
\[
V\,:\,\comodules{{\mathfrak A}^\circ}\,\longrightarrow\,\comodules{(\mathbb{C}\Gamma)^\circ}
\]
by means of a composition as indicated in
\[
\xymatrix{
{\mathcal T}(M)\ar[rr]^\iota&&{\mathcal C}_{\mathrm{dR}}(M)\ar[d]_\sim^{\mathrm{Mon}^{x_0}}
\\
\modules{\mathfrak A}\ar[u]_\sim^{\mathcal{V}} && \modules{\mathbb{C}\Gamma}\ar[d]_\sim^{\tau} 
\\
\comodules{{\mathfrak A}^\circ}\ar[rr]_V\ar[u]_\sim^\sigma&& \comodules{(\mathbb{C}\Gamma)^\circ}. 
}
\]
Here, $\tau$ and $\sigma$ are the equivalences mentioned in \eqref{20.09.2022--1}. 
Letting $\omega_\star$ denote ``the'' forgetful functor, we have a natural tensor isomorphism
\[\eta\,:\,
\omega_{\mathfrak{A}^\circ}\,\stackrel{\sim}{\longrightarrow}\,\omega_{(\mathbb{C}\Gamma)^\circ}\circ V
\]
constructed from the canonical isomorphisms
\[\omega_{\mathfrak A}\,\stackrel{\sim}{\longrightarrow}\, \bullet|_{x_0}\circ\mathcal{V}
\quad\text{ and }\quad
\omega_{\mathbb C\Gamma}\circ\mathrm{Mon}^{x_0}\,\stackrel{\sim}{\longrightarrow}\, \bullet|_{x_0}.
\]
(It should be noted that $\omega_{\mathfrak A}\circ\sigma\,=\,\omega_{\mathfrak{A}^\circ_M}$ and 
$\omega_{(\mathbb C\Gamma)^\circ}\circ\tau\,=\,\omega_{\mathbb C\Gamma}$.) In simpler terms, given 
$E\in\comodules{{\mathfrak A}^\circ}$, the vector space $V(E)$ is $\mathcal 
V(E)|_{x_0}\,=\,(\mathcal O_M\otimes E)|_{x_0}$ and $\eta_E$ is the natural isomorphism $E\,
\stackrel{\sim}{\longrightarrow} \mathcal V(E)|_{x_0}$.

We obtain an arrow of \emph{Hopf algebras}
\[
v\,:\, {\mathfrak A}^\circ\,\longrightarrow\,(\mathbb{C}\Gamma)^\circ
\]
jointly with an isomorphism of tensor functors 
\[
\overline\eta\,:\,v^\#\,\longrightarrow \,V
\] 
determined by $\eta$ \cite[II.2.1.2.1-3, p.116]{Sa}. 
That is, $\eta$ equals the composition
\[
\omega_{{\mathfrak A}^\circ}\,=\,\omega_{(\mathbb{C}\Gamma)^\circ}\circ v^\# \xymatrix{\ar[rr]^-{\omega_{(\mathbb{C}\Gamma)^\circ}\star\overline\eta}&&} \omega_{(\mathbb{C}\Gamma)^\circ}\circ V.
\]
Again, in simpler terms, given $E\in\comodules{\mathfrak{A}^\circ}$, the canonical isomorphism 
of vector spaces $\eta_E\,:\,E\,\stackrel{\sim}{\longrightarrow} \, \mathcal V(E)|_{x_0}$ is an isomorphism of 
$(\mathbb{C}\Gamma)^\circ$--comodules, where $\mathcal V(E)|_{x_0}$ carries the co-action coming 
from the monodromy and $E$ the co-action coming from $v$.

Applying the results of Section \ref{20.09.2022--2} to $v$, we derive a continuous morphism of topological $\mathbb C$--algebras 
\[
u\,\,:\,\,\mathbb{C}\Gamma\,\widehat{\,\,}\,\longrightarrow\, \widehat{\mathfrak{A}}
\]
which makes the analogue diagram in Lemma \ref{20.09.2022--3} commute strictly. We can then say 
the following. Let $E\in\comodules{{\mathfrak A}^\circ}$ be given. Endow $E$, respectively 
$\mathcal{V}(E)|_{x_0}$, with its $\mathbb{C}\Gamma$--module structure coming from $u$, 
respectively from the monodromy. Then $\eta\,:\,E\,\longrightarrow \,\mathcal V(E)|_{x_0}$ is an arrow of $\mathbb 
C\Gamma$--modules.

\subsection{Description of $u:\mathbb C\Gamma\,\widehat{\,\,}\to\widehat{\mathfrak A}$ in
terms of iterated integrals}

We can now explicitly describe the action of $\Gamma$ on ${\mathfrak A}$--modules and hence the 
morphism of $\mathbb C$--algebras $u\,:\,\mathbb C\Gamma\,\widehat{\,\,}
\,\longrightarrow \,\widehat{\mathfrak A}$.

Let $R$ be a finite dimensional associative $\mathbb{C}$--algebra and $f\,:\,{\mathfrak A}
\,\longrightarrow \,R$ a surjective morphism. Regarding $R$ as a left ${\mathfrak A}$--module, let 
$\mathrm d_f$ be the connection on $\mathcal O_M\otimes R$ defining $\mathcal V(R)$
and
\[
\varphi_f\,= \,{\rm d}_f(1\otimes1)\,\in\, H^0(M,\, \Omega_M^1)\otimes R.
\]
 
\begin{proposition}\label{01.09.2020--1}
For any smooth loop $\gamma$ based at $x_0$, the series
\[1+ \int_\gamma\varphi_f +\int_\gamma\varphi_f^2 +\ldots\]
in $R$ converges, with respect to the \emph{Euclidean} topology, to $f(u(\gamma))$.
\end{proposition}

\begin{proof} Let $\lambda\,:\,R\,\longrightarrow \,\mathrm{End}(R)$
be the morphism of left multiplication. 
With this notation,
\[\Phi_f\,:=\,( \mathrm{id}\otimes\lambda)(\varphi_f)\,\in\,H^0(M,\, \Omega_M^1) \otimes\mathrm{End}(R)
\]
is the 1--form associated to the connection ${\rm d}_f$ (see \eqref{16.09.2022--1}). 
As seen in \eqref{27.08.2020--1}, the series
\[1+
\sum_{n\ge1} \int_\gamma\Phi_f^n
\]
converges to $\mathrm{Mon}_{\gamma}\,\in\, \mathrm{End}(R)$. (We write $\mathrm{Mon}_{\gamma}$
instead of $\mathrm{Mon}^{x_0}_{\mathrm d_{\Phi_f}}(\gamma)$.) 
Now, for each $n\,\in\,\mathbb{N}$, we have in $\mathrm{End}(R)$ that
\[
\lambda\int_\gamma\varphi_f^n \,=\,\int_\gamma\Phi_f^n,
\] 
so that $1+\int_\gamma\varphi_f +\int_\gamma\varphi_f^2+\ldots$ converges to an
element $m_\gamma$ which satisfies $\lambda (m_\gamma)\,=\,
{\rm Mon}_\gamma$. In particular, $\mathrm{Mon}_\gamma(1_R)\,=\,m_\gamma$. 
But, by construction of $u\,:\,\mathbb{C}\Gamma\,\widehat{\,\,}\,\longrightarrow
\,\widehat{\mathfrak{A}}$ (see
\eqref{23.09.2022--1} and Section \ref{20.09.2022--4}), we know that
\[
\begin{split}
{\rm Mon}_\gamma(1_R)&\,=\,\,f(u(\gamma))\cdot1_R ,
\end{split}
\]
so that $m_\gamma\,=\,f(u(\gamma))$.
\end{proof}

As a consequence, we have the following. Let $\gamma$ be a smooth loop based at $x_0$ and
define, for every $n\,\in\,\mathbb N\setminus\{0\}$, the linear functional
\[J_{\gamma,n}\,:\,H^0(M,\, \Omega^1_N)^{\otimes n}\, \longrightarrow\,\mathbb{C}\]
by 
\[
\varphi_1\otimes\cdots\otimes\varphi_n\,\longmapsto \,\int_\gamma \varphi_1\cdots\varphi_n.
\]
This gives rise to an element of $[H^0(M,\, \Omega^1_M)^*]^{\otimes n}$, then to an element of
${\mathfrak A}$ and hence to an element \[J_{\gamma,n}\,\in\,\widehat{\mathfrak A}.\] For
convenience, we decree that $J_{\gamma,0}\,=\,1$. By fixing a basis $\{\theta_i\}_{i=1}^g$
of $H^0(M,\, \Omega^1_M)$, a dual basis $\{t_i\}_{i=1}^g$ and writing $t_{i_1,\ldots,i_n}$ for
$t_{i_1}\otimes\cdots\otimes t_{i_n}$, we have 
\[
J_{\gamma,n} \,=\, \sum_{1\le i_1,\ldots,i_n\le g}\left( \int_\gamma \theta_{i_1}\cdots\theta_{i_n} \right) \cdot t_{i_1,\ldots,i_n}\quad \begin{array}{c}\text{modulo integrability} \\ \text{relations}.\end{array}
\]
These elements were studied firstly by Chen and Parshin \cite{parshin} and hence we call the arrow 
\[\mathrm{cp}_N\,:\, \Gamma\,\longrightarrow\,\widehat{\mathfrak A},\quad
\gamma\,\longmapsto\,\sum_{n=0}^N J_{\gamma,n}
\]
the Chen-Parshin maps of order $N$ (as of now, these are maps of sets) and $(\mathrm{cp}_N)$ the Chen-Parshin sequence. 

\begin{lemma}\label{17.08.2020--1}Let $R$ be a finite dimensional associative $\mathbb{C}$--algebra
and $f:{\mathfrak A}\,\longrightarrow\, R$ a surjection. Let ${\rm d}_f$ stand for the connection
on ${\mathcal O}_M\otimes R$ associated to the left ${\mathfrak A}$--module structure on $R$, and
let $\varphi_f\in R\otimes H^0(M,\, \Omega_M^1)$ be the 1--form ${\rm d}_f(1\otimes1)$. Then,
for each $n\,\in\,\mathbb{N}$, we have
\[
f( J_{\gamma,n})\,\, =\,\, \int_\gamma\varphi_f^n.
\]
\end{lemma}

\begin{proof}[Sketch proof.] Let $\varphi_f\,=\,\sum_{i=1}^ge_i\otimes\theta_i$ and denote the function $\langle\theta_i,\gamma'\rangle$ by $a_i$. It follows that $\langle \varphi_f,\gamma'\rangle = \sum_ie_ia_i\in C^\infty([0,1], R)$. 
Consequently, 
\[\begin{split}
\int_\gamma\varphi_f^n& = \int_{0\le s_1\le\cdots\le s_n\le1} \left(\sum_{1\le i_1\le g}e_{i_1}a_{i_1}(s_1)\right)\cdots \left(\sum_{1\le i_n\le g}e_{i_1}a_{i_n}(s_n)\right)\, ds_1\cdots ds_n
\\
&=\sum_{1\le i_1,\ldots,i_n\le g} \int_{0\le s_1\le\cdots\le s_n\le1} e_{i_1}\cdots e_{i_n}a_{i_1}(s_1)\cdots a_{i_n}(s_n)\,ds_{1}\ldots ds_n
\\
&=\sum_{1\le i_1,\ldots,i_n\le g}e_{i_1,\ldots,i_n } \int_{\gamma}\theta_{i_1}\cdots\theta_{i_n}.
\end{split}
\]\end{proof}

For each $I\in S_{\mathfrak A}$, endow the complex vector space ${\mathfrak A}/I$ with the {\it canonical Euclidean topology} and give $\widehat{\mathfrak A}$ the projective limit (product) topology. This topology shall be called the PE-topology (projective-Euclidean). 
Hence, a sequence $(\alpha_n)$ in $ \widehat{\mathfrak A}$ converges to $\alpha\in\widehat{\mathfrak A}$ if and only if, for each $I\in S_{\mathfrak A}$, the induced sequence $(\alpha_{n,I})$ in ${\mathfrak A}/I$ converges $\alpha_I$ \cite[4 Theorem, p .91]{kelley55}. 

\begin{corollary}In $\widehat{\mathfrak{A}}$, we have the equality \[\sum_{n=0}^\infty
J_{\gamma,n}\,=\,u(\gamma),
\] 
where convergence is taken in the PE-topology. In other words, $u(\gamma)$ is the limit of the Chen-Parshin sequence $(\mathrm{cp}_N(\gamma))$
in the PE topology.
\end{corollary}

\begin{proof}Let $I\in S_{\mathfrak A}$ and denote by $f\,:\,{\mathfrak A}\,
\longrightarrow\,R$ the quotient morphism. As in Lemma \ref{17.08.2020--1}, we
write $\varphi_f\,=\,\mathrm d_f(1\otimes1)$. Then,
\[
f\left(\sum_{n=0}^NJ_{\gamma,n}\right)\,\,=\,\,\sum_{n=0}^N\int_\gamma\varphi_f^n
\] according to Lemma \ref{17.08.2020--1}. Since the sequence $N\,\longmapsto\,
\sum_{n=0}^N\int_\gamma\varphi_f^n$ converges (Proposition \ref{01.09.2020--1}),
the definition of the product topology proves that $\sum_{n\ge0}\,J_{\gamma,n}$ converges in $\widehat{\mathfrak{A}}$. The limit is $u(\gamma)$. 
\end{proof}


\begin{thebibliography}{000}
\bibitem{ABDH} M. Aprodu, I. Biswas, S. Dumitrescu and S. Heller, On the monodromy map for logarithmic systems, to be published 
{\it Bul. Soc. Math. Fr} {\bf 150} (2022), 543-568.

\bibitem{At} M. F. Atiyah, Complex analytic connections in fibre bundles,
{\it Trans. Amer. Math. Soc.} {\bf 85} (1957), 181--207.

\bibitem{biswas97} I. Biswas, Parabolic bundles as orbifold bundles, {\it Duke Math. Jour.} {\bf 88} (1997), 305--326.

\bibitem{biswas-hai-dos_santos22} I. Biswas, P. H. Hai and J. P. dos Santos, Connections on Trivial Vector Bundles over 
Projective Schemes. Preprint October 2021.

\bibitem{Bi} I. Biswas, \'Etale triviality of finite vector bundles over compact complex manifolds,
{\it Adv. Math.} {\bf 369} (2020), 107167.

\bibitem{BDDH} I. Biswas, J. P. dos Santos, S. Dumitrescu and S. Heller, On certain Tannakian categories
of integrable connections over K\"ahler manifolds, {\it Canad. J. Math.} {\bf 74} (2022), 1034--1061.

\bibitem{BD} I. Biswas and S. Dumitrescu, Holomorphic bundles trivializable by proper surjective
holomorphic map, {\it Int. Math. Res. Notices} {\bf 5} (2022), 3636--3650. 

\bibitem{BD2} I. Biswas and S. Dumitrescu, Riemann-Hilbert correspondence for 
differential systems over Riemann surfaces, {\em Publ. Res. Inst. Math. Sci.}
{\bf 59} (2023).

\bibitem{BDH} I. Biswas, S. Dumitrescu and S. Heller, Irreducible flat ${\rm 
SL}(2,\mathbb R)$--connections on the trivial holomorphic bundle, {\em J. Math. Pures 
Appl.} {\bf 149} (2021), 28--46.

\bibitem{BDHH} I. Biswas, S. Dumitrescu, L. Heller and S. Heller, Fuchsian $sl(2,\mathbb C)$--systems of compact Riemann surfaces 
(with an appendix by Takuro Mochizuki), {\em arxiv.org/abs/2104.04818.}

\bibitem{BDHH2} I. Biswas, S. Dumitrescu, L. Heller and S. Heller, On the existence of holomorphic curves in compact quotients of 
$\text{SL}(2, \mathbb C)$, {\em arXiv.org/abs/2112.03131.}

\bibitem{BdS}I. Biswas and J. P. dos Santos, Vector bundles trivialized by proper morphisms and the fundamental group scheme,
{\it Jour. Inst. Math. Jussieu} {\bf 10} (2011), 225--234.

\bibitem{Bo} A. Borel, P.-P. Grivel, B. Kaup, A. Haefliger, B. Malgrange and F. Ehlers,
{\rm Algebraic $D$--modules}, Perspectives in Mathematics, 2. Academic Press, Inc., Boston, MA, 1987.

\bibitem{bourbaki-general-topology} N. Bourbaki, {\it Elements of Mathematics} (General Topology. Chapters 1--4. Springer. 1995).

\bibitem{chen71}K.-T. Chen, Algebras of Iterated Path Integrals and the Fundamental Group, {\it Trans. Amer. Math. Soc.}
{\bf 156} (1971), 359--379.

\bibitem{CDHL} G. Calsamiglia, B. Deroin, V. Heu and F. Loray, The Riemann-Hilbert mapping for 
$sl(2)$--systems over genus two curves, {\em Bull. Soc. Math. France} {\bf 147} (2019), 159--195.

\bibitem{coddington61}E. Coddington, {\it An introduction to ordinary differential equations} (Prentice-Hall Mathematics Series 
Prentice-Hall, 1961).

\bibitem{DMOS} P. Deligne, J. S. Milne, A. Ogus and K.-y. Shih,
{\it Hodge cycles, motives, and Shimura varieties},
Lecture Notes in Mathematics, 900, Springer-Verlag, Berlin-New York, 1982.

\bibitem{Gh} E. Ghys, D\'eformations des structures complexes sur les espaces homog\`enes de 
${\rm SL}(2, \mathbb{C})$, {\em Jour. Reine Angew. Math.} \textbf{468} (1995), 113--138.

\bibitem{hain86} 
R. M. Hain, On a generalization of Hilbert's 21st problem, {\it Ann. Sci. \'Ecole Norm. Sup.} {\bf 19} (1986), 609--627. 

\bibitem{HHSch} L. Heller, S. Heller and N. Schmitt, Navigating the space of symmetric cmc surfaces, {\em Jour. Diff.
Geom.} {\bf 110} (2018), 413--455.

\bibitem{hochschild-mostow57} G. Hochschild and G. D. Mostow, Pro-affine
algebraic groups, {\it Amer. Jour. Math.} {\bf 91} (1969), 1141--1151.

\bibitem{HM} A. H. Huckleberry and G. A. Margulis, Invariant Analytic hypersurfaces, 
{\em Invent. Math.} {\bf 71} (1983), 235--240.

\bibitem{HW} A. H. Huckleberry and J. Winkelmann, Subvarieties of parallelizable manifolds, {\em Math. Ann.} {\bf 295} (1993), 469--483.

\bibitem{Ka} N. M. Katz, An overview of Deligne's work on Hilbert's twenty-first problem,
{\it Mathematical developments arising from Hilbert problems}
(Proc. Sympos. Pure Math., Vol. XXVIII, Northern Illinois Univ., 
De Kalb, Ill., 1974), pp. 537--557. Amer. Math. Soc., Providence, R. I., 1976.

\bibitem{kelley55} Kelley, {\it General Topology} (Graduate Texts in Mathematics, No. 27. Springer-Verlag, 1975).

\bibitem{La} M. Lackenby, Some 3-manifolds and 3-orbifolds with large fundamental group,
{\it Proc. Amer. Math. Soc.} {\bf 135} (2007), 3393--3402. 

\bibitem{Loray16} F. Loray, Isomonodromic deformations of Lame connections, the Painleve 
VI equation and Okamoto symmetry, {\em Izv. Ros. Acad. Nauk. Ser. Math} 80(1) (2016), 
119--176.

\bibitem{montgomery93}S. Montgomery, {\it Hopf Algebras and Their Actions on Rings}. CBMS Reg. Conf. Ser.
Math. 82. Amer. Math. Soc., Providence, RI.

\bibitem{No1} M. V. Nori, On the representations of the fundamental group,
\textit{Compositio Math.} {\bf 33} (1976), 29--41.

\bibitem{No2} M. V. Nori, The fundamental group--scheme, \textit{Proc. Ind. Acad. Sci.
(Math. Sci.)} \textbf{91} (1982), 73--122.

\bibitem{parshin}A. N. Parshin, {\it A generalization of the Jacobian variety}. Amer. Math. Soc. Transl.
Second Series 84 (1969), 187--196.

\bibitem{Ra} M. S. Raghunathan, Vanishing theorems for cohomology groups associated to discrete subgroups of semi-simple Lie 
groups, {\it Osaka J. Math.} {\bf 3} (1966), 243--256, corrections ibid. {\bf 16} (1979), 295--299.

\bibitem{Raj} C. S. Rajan, Deformations of complex structures on $\Gamma \backslash \text{SL}(2, \mathbb C$ and cohomology of local 
systems, {\it Proc. Indian Acad. Sci. (Math. Sci.)} {\bf 104} (1994), 389--395.

\bibitem{Sa} N. Saavedra Rivano, {\it Cat\'egories Tannakiennes}
(Lecture Notes in Mathematics, Vol. 265, Springer-Verlag, Berlin-New York, 1972).

\bibitem{SGA3} Sch\'emas en groupes (SGA 3). Tome I. {\it Propri\'et\'es g\'en\'erales des sch\'emas en groupes.} Doc. Math. (Paris), 
7. Soc. Math. France, Paris, 2011.

\bibitem{Sik} A. S. Sikora, Character varieties, {\it Trans. Amer. Math. Soc.}
{\bf 364} (2012), 173--208.

\bibitem{Si} C. T. Simpson, Higgs bundles and local systems,
{\it Inst. Hautes \'Etudes Sci. Publ. Math.} \textbf{75} (1992), 5--95.

\bibitem{sweedler69} M. Sweedler, {\it Hopf algebras} (New York: W.A. Benjamin, Inc., 1969).

\bibitem{Wa} H.-C. Wang, Complex Parallelisable manifolds, {\em Proc. 
Amer. Math. Soc.} \textbf{5} (1954), 771--776.

\end{thebibliography}
\end{document}